\newcommand*{\mailto}[1]{\href{mailto:#1}{\nolinkurl{#1}}}

\newcommand{\CT}{\mathcal{T}}
\newcommand{\RR}{\mathbb{R}}
\newcommand{\QQ}{\mathbb{Q}}
\newcommand{\CS}{\mathcal{S}}

\newcommand{\CO}{\mathcal{O}}
\newcommand{\CP}{\mathcal{P}}
\newcommand{\BF}{\mathbb{F}}
\newcommand{\CF}{\mathcal{F}}
\newcommand{\SF}{\mathscr{F}}
\newcommand{\CM}{\mathcal{M}}
\newcommand{\NN}{\mathbb{N}}

\newcommand{\PP}{\mathbb{P}}
\newcommand{\WienH}{\mathcal{H}}
\newcommand{\MA}{\mathfrak{A}}
\newcommand{\INT}{[0,T]}
\newcommand{\EE}{\mathbb{E}}
\newcommand{\CB}{\mathcal{B}}
\newcommand{\CL}{\mathcal{L}}
\newcommand{\lk}{\left}
\newcommand{\rk}{\right}
\newcommand{\la}{\langle}
\newcommand{\ra}{\rangle}
\newcommand{\CH}{\mathcal{H}}

\newcommand{\Sn}{\{S_n\}}

\newcommand{\DD}{\mathbb{D}}
\newcommand{\bW}{\mathbb{W}}
\newcommand{\CK}{\mathcal{K}}
\newcommand{\Law}{\operatorname{Law}}
\newcommand{\CI}{\mathcal{I}}
\newcommand{\CMM}{\mathcal{M}}

\newcommand{\abs}[1]{\left\lvert#1\right\rvert}
\newcommand{\norm}[1]{\left \lVert#1 \right\rVert}

\newcommand\Leb{\operatorname{Leb}}
\newcommand{\distr}{\mu}
\newcommand{\vast}{{V'}}
\newcommand{\hast}{{H'}}
\newcommand{\kk}{k}

\newcommand{\DEQS}{\begin{eqnarray*}}
	\newcommand{\EEQS}{\end{eqnarray*}}
\newcommand{\DEQSZ}{\begin{eqnarray}}
	\newcommand{\EEQSZ}{\end{eqnarray}}

\documentclass[10pt,reqno]{amsart}
\usepackage{amsmath}
\usepackage{amssymb}
\usepackage{amsfonts}
\usepackage{amsthm}
\usepackage{mathrsfs}
\usepackage{dsfont}
\usepackage{mathtools}
\usepackage{todonotes}
\usepackage{enumitem}
\usepackage{xargs}

\newcommand\del[1]{}

\newcommandx{\intodofahim}[2][1=]{\todo[linecolor=red,backgroundcolor=red!25,bordercolor=red,#1]{#2}}

\usepackage{xcolor}
\definecolor{darkgreen}{rgb}{0.5,0.25,0}
\definecolor{darkblue}{rgb}{0,0,1}
\definecolor{answerblue}{rgb}{0,0,0.75}
\definecolor{DarkGreen}{rgb}{0.1,0.6,0.2}   
\definecolor{Yellow}{rgb}{0.7,0.7,0.7}   

\usepackage[backref=page]{hyperref}
\hypersetup{colorlinks,breaklinks,
linkcolor=darkblue,urlcolor=darkblue,
anchorcolor=darkblue,citecolor=darkblue}

\makeatletter
\@namedef{subjclassname@2020}{%
  \textup{2020} Mathematics Subject Classification}
\makeatother

\newtheorem{theorem}{Theorem}[section]
\newtheorem{lemma}[theorem]{Lemma}

\newtheorem{definition}[theorem]{Definition}

\newtheorem{example}[theorem]{Example}

\newtheorem{remark}[theorem]{Remark}

\newtheorem{hypo}[theorem]{Assumption}

\numberwithin{equation}{section}
\allowdisplaybreaks

\begin{document}

\title[A general Yamada-Watanabe result]
{Yamada-Watanabe uniqueness results for SPDEs driven 
by Wiener and pure jump processes}

\author[Fahim]{Kistosil Fahim}
\address[Kistosil Fahim]
{Department of Mathematics,
Institut Teknologi Sepuluh Nopember,
Kampus ITS Sukolilo,
Surabaya 60111, Indonesia}
\email{\mailto{kfahim@matematika.its.ac.id}}

\author[Hausenblas]{Erika Hausenblas}
\address[Erika Hausenblas]
{Department of Mathematics,
Montanuniversity Leoben,
8700 Leoben, Austria}
\email{\mailto{erika.hausenblas@unileoben.ac.at}}

\author[Karlsen]{Kenneth H. Karlsen}
\address[Kenneth Hvistendahl Karlsen]
{Department of Mathematics,
University of Oslo,
NO-0316 Oslo, Norway}
\email{\mailto{kennethk@math.uio.no}}



\thanks{This work was supported by the
project Pure Mathematics in Norway, funded by
Trond Mohn Foundation and Troms{\o} 
Research Foundation, and by the Research Council 
of Norway under project 351123 (NASTRAN). 
Erika Hausenblas and Kistosil Fahim was supported by the 
ASEA network of the Austrian Exchange Service.}

\date{\today}

\begin{abstract}
The Yamada-Watanabe theory
provides a robust framework for understanding
stochastic equations driven by Wiener processes.
Despite its comprehensive treatment in the literature,
the applicability of the theory to SPDEs driven by
Poisson random measures or, more generally, L{\'e}vy processes
remains significantly less explored,
with only a handful of results addressing this context. 
In this work, we leverage a result by Kurtz to 
demonstrate that the existence
of a martingale solution combined with 
pathwise uniqueness implies the existence 
of a unique strong solution for SPDEs driven by both a
Wiener process and a Poisson random measure.
Our discussion is set within the variational 
framework, where the SPDE
under consideration may be nonlinear.
This work is influenced by earlier research 
conducted by the second author alongside 
de Bouard and Ondrej\'at.
\end{abstract}

\maketitle

\setcounter{secnumdepth}{2}
\setcounter{tocdepth}{2}


\section{Introduction}
The development of the Yamada-Watanabe uniqueness
theory for stochastic equations primarily focuses on those
driven by Wiener processes. The literature becomes
more sparse when considering equations influenced by both
Wiener noise and Poisson random measures,
particularly in the infinite-dimensional
noise case. Notably, in \cite{poisson1}, 
the theory is extended to stochastic
differential equations are driven by this dual 
noise setup on a locally compact space,
using the original Yamada and Watanabe method. 
Similarly, \cite{poisson3} applies the theory to 
variational solutions of PDEs driven solely 
by a Poisson random measure, again 
on a locally compact space, employing the same 
basic methods. Reference \cite{Bouard:2019ab} presents
the theory within a semigroup framework, focusing exclusively on
a Poisson random measure.

Our research aims to broaden this scope by investigating
the Yamada-Watanabe theory on Banach spaces,
accommodating both Poisson random measures
and Wiener processes in infinite dimensions within
a variational setting.

In their foundational work \cite{Yamada:1971ab,Watanabe:1971aa},
Yamada and Watanabe presented a proof leveraging the concept
of a regular version of conditional probabilities.
Their methodology has since proven to be exceptionally robust
and versatile, finding applications well beyond stochastic 
differential equations. The development of an abstract Yamada-Watanabe
theory by Kurtz in 2007 \cite{Kurtz:2007aa} marked an important
expansion of the theory to a broad array of stochastic problems.
Kurtz's original argument, rooted in the Skorokhod
representation theorem, abstracted the Yamada-Watanabe principle
to a new level of generality. It was further explored by
Kurtz in \cite{Kurtz:2014aa}.

Applying Kurtz's abstract framework, 
Bouard, Hausenblas, and Ondrejat in \cite{Bouard:2019ab} 
established that for stochastic evolution
equations driven by a Poisson random measure, the pathwise
uniqueness and the existence of a martingale solution
implies the existence of a unique, strong solution.
Our work seeks to advance this line of inquiry by
adapting these concepts to the variational framework and
by incorporating also a Wiener process into the analysis. 
Our initial motivation arose from the requirement 
in \cite{Hausenblas:2024aa} to establish specific uniqueness 
results for a L{\'e}vy-driven stochastic bidomain model in 
electrophysiology. However, given its broad applicability, 
we believe that a readily citable result of 
this kind would be valuable in different contexts, which 
led us to write this paper.

The remainder of this paper is structured as follows: In 
Section \ref{sec:prelim}, we review key concepts from stochastic 
analysis, including cylindrical Wiener processes, Poisson random 
measures, and L{\'e}vy processes. Sections \ref{kframework} 
and \ref{sec:uniqueness} are dedicated to the development of 
the Yamada-Watanabe theory. In Section \ref{kframework}, we 
explain Kurtz's abstract framework within the context of SPDEs, 
while Section \ref{sec:uniqueness} demonstrates how 
weak existence and pathwise uniqueness 
together imply uniqueness in law.


\section{Preliminary material}\label{sec:prelim}

Before going into the stochastic preliminaries, let us first
establish some notations that will be utilized throughout this paper.
We denote the set of real numbers as $\RR$, with
$\RR^+:=\{ x\in\RR:x>0\}$ and $\RR^+_0:=\RR^+\cup\{0\}$.
The set of natural numbers, including $0$, 
is denoted by $\mathbb{N}$.
If $(\mathcal F_t)_{t\in [0,T]}$ represents a filtration and
$\theta$ is a measure, we use $\mathcal F_t^\theta$ to
denote the augmentation of $\mathcal F_t$ with the $\theta$-null sets
contained in $\mathcal F_T^{\theta}$.

A measurable space $(S,\CS)$ is called \textit{Polish}
if there exists a metric $\varrho$ on $S$ 
such that $(S,\varrho)$ is a complete 
separable metric space and 
$\CS=\mathscr{B}(S)$ (i.e., $\CS=$ the Borel sets 
of $S$ with respect to $\varrho$).
For a Polish  space $(S,\CS)$, we denote by 
$\mathcal{B}(S)$ the set of all Borel measurable 
mappings $F:S\to\RR$ for which
$F$ is $\mathscr{B}(S)/\mathscr{B}(\RR)$ measurable. 
The collection of all finite non-negative 
measures on a Polish space $(S,\CS)$ is
denoted by $M^+(S)$. Additionally, $\CP(S)$
represents the set of probability measures on $S$.
As $S$ is a separable metric space, $\CP(S)$ can be 
metrized as a separable metric 
space \cite[Theorem 6.2, page 43]{Parthasarathy:1967ul}. 
Moreover, $\CP(S)$ is a compact metric space 
if and only if $S$ is compact 
\cite[Theorem 6.4, page 45]{Parthasarathy:1967ul}.

If a family of sets $\{S_n\in\CS:n\in\NN\}$ 
satisfies $S_n\uparrow S$, then
$M_{ \NN}(\{S_n\})$ denotes the family of all
$\NN\cup \{\infty\}$-valued measures $\mu$ on $\CS$ 
such that $\mu(S_n)<\infty$ for every $n\in\Bbb N$.
By $\CM_\NN(\{S_n\})$, we denote the $\sigma$-algebra 
on $M_{\NN}(\{S_n\})$ generated by the functions
$i_B:M_\NN(\{S_n\})\ni\mu \mapsto 
\mu(B)\in \NN$, $B\in \CS$.
The following simple result can be proven 
directly by constructing the metric.

\begin{lemma}\label{measure_lemma}
Let $(S,\CS)$ be a Polish space and consider a
family $\{S_n\in\CS\}$ of sets 
satisfying $S_n\uparrow S$. Then
$(M_\NN(\{S_n\}),\CM_\NN(\{S_n\}))$
is a Polish space.
\end{lemma}

\begin{proof}
Fix any $n\in\NN$. Let $\mu \in M^+_{\NN}(S_n)$ be a 
mapping $\mu:\mathscr{B}(S_n) \to \NN$ 
that is measurable and satisfies $\mu(S_n) < \infty$, 
where the $+$ superscript indicates the 
positive elements in $M_{\NN}(S_n)$.
Define $F := \{f_n \in C(S) : n \in \NN\}$
as a family of functions such that $|f_n|_{C(S)} \leq 1$ 
and $F$ separates points in $S$. Here, $C(S)$ denotes the 
set of continuous real-valued functions defined on $S$, 
equipped with the supremum norm $|f|_{C(S)}$. 
By the Hahn-Banach Theorem, and 
given that $\CP(S)$ is separable, 
there exists such a countable set that 
separates points. Additionally, let $\{\lambda_n : n \in \NN\}$ be 
such that $\sum_{n \in \NN} \lambda_n < \infty$. 
Now, define the following metric for 
$\mu_1, \mu_2 \in M_{\NN}^+(S_n)$:
$$
d_{S_n}(\mu_1, \mu_2) := \sum_{n \in \NN} \lambda_n 
\frac{ \left| \langle \mu_1,f_n\rangle 
- \langle\mu_2,f_n\rangle \right|}
{1 + \left| \langle \mu_1,f_n\rangle 
- \langle\mu_2,f_n\rangle \right|}.
$$
This defines a metric on $M_{\NN}^+({S_n})$. 
For a measure $\mu: \CM(\{S_n\}) \to \RR_0^+$, define 
$\mu|_{S_n}(A) := \mu(A \cap S_n)$ for $A \in \mathscr{B}(S)$. 
Next, we define a metric on $M_{\NN}^+(\{S_n\})$:
$$
d_{S}(\mu_1, \mu_2) := \sum_{n \in \NN} \lambda_n 
\frac{d_{S_n}(\mu_1|_{S_n}, \mu_2|_{S_n})}
{1 + d_{S_n}(\mu_1|_{S_n}, \mu_2|_{S_n})}, 
\quad \mu_1, \mu_2 \in M_{\NN}^+(\{S_n\}).
$$
It is straightforward to verify that $d_S$ defines 
a metric on $M_{\NN}^+(\{S_n\})$. 
It remains to show that this metric is complete and 
that $M_\NN(\{S_n\})$ is separable.

Note that if $\mu^k, \mu \in M_{\NN}^+(\{S_n\})$ 
and $\mu^k \to \mu$ in the weak topology, 
then we have $\mu^k|_{S_n}(f) \to \mu|_{S_n}(f)$ for each $f$. 
Hence, $d_{S_n}(\mu^k|_{S_n}, \mu|_{S_n}) \to 0$, 
and therefore $d_S(\mu^k, \mu) \to 0$. 
Finally, separability follows again from the 
Hahn-Banach Theorem.
\end{proof}

Let us transition to the stochastic framework.
All stochastic processes are defined on 
a filtered probability space. 
Throughout the paper, we denote this 
filtered probability space by 
\begin{equation}\label{probabilitysp}
	\MA=(\Omega,\CF,\BF, \PP),
\end{equation}
where $\PP$ is complete on $(\Omega, \CF)$, and 
$\BF=(\CF_t)_{t\in [0,T]}$ is a filtration 
satisfying the \textit{usual conditions}:
\begin{enumerate}
	\item[(i)] for each $t\in [0,T]$, $\CF_t$
	contains all $(\CF,\PP)$-null sets;
	
	\item[(ii)] the filtration $\mathbb{F}$ 
	is right-continuous.
\end{enumerate}

A L\'evy process encompasses both 
a cylindrical Wiener process and a pure jump process, 
for both of which we will provide detailed definitions. 
To understand the characteristics of a pure jump process, 
we utilize Poisson random measures. 

\subsection{The cylindrical Wiener process}\label{defcyw}

Let $\WienH$ denote a separable Hilbert space. 
Consider $\bW$ as a cylindrical Wiener process
evolving over $\WienH$, defined on $\MA$. By using
the spectral decomposition theorem, the Wiener
process can be expressed as
$$
\bW(t) := \int_0^t \sum_{k=1}^\infty h_k
\, d\beta_k(t), \quad t \ge 0,
$$
where $\{\beta_{k} : k \in \mathbb{N}\}$ is a family
of mutually independent Brownian motions,
and $\{h_k : k \in \mathbb{N}\}$
forms an orthonormal basis in $\WienH$. Referencing
Proposition 4.7 in \cite[p.~85]{DaPrato:2014aa}, we 
note that this representation can be considered 
without loss of generality.

To facilitate our later discussions, let us introduce the concept
of Hilbert-Schmidt operators between two Hilbert spaces
$\CH$ and $H$.
We define the space of all Hilbert-Schmidt operators
$\mathcal{L}_{\text{HS}}(\CH,H)$ as follows:
\begin{equation*}
	\mathcal{L}_{\text{HS}}(\CH,H)
	:=\left\{ L:\CH\to H: \sum _{i=1}^\infty \abs{Le_{i}}^{2}_{H}
	=\sum _{i,k=1}^\infty \left|\left\langle
	Le_{i},f_{k} \right\rangle\right |^{2}<\infty\right\},
\end{equation*}
where $(e_i)$ and $(f_i)_i$ denote orthonormal bases in
$\CH$ and $H$, respectively.

\subsection{Poisson random measures 
and L{\'e}vy processes}\label{deflev}

To begin, we revisit \cite[Definition I.8.1]{Ikeda:1981aa}.
Given the diverse and not always consistent approaches
to defining a time-homogeneous Poisson random measure
in the literature, we present our own definition
for clarity and precision in this context.

\begin{definition}[{\cite[Definition I.8.1]{Ikeda:1981aa}}]
\label{def-PRM}
Let $(S,\CS)$  be a Polish space and
$\nu$ be a $\sigma$--finite measure on $(S,\CS)$, where
$\sigma$--finite  means that there exists a nested
sequence $\{S_n\in\CS\}_{n\in\NN}$
such that $S_n\uparrow S$ and $\nu(S_n)<\infty$
for every $n\in\Bbb N$.

A \textit{time homogenous Poisson random measure} $\eta$
over a filtered  probability space $(\Omega,\CF,\BF,\PP)$,
where $\BF=(\CF_t)_{t\in\INT}$, is a measurable function
$$
\eta: (\Omega,\CF)\to \bigl(M_{\NN}(\{S_n \times\INT\}),
\mathcal{M}_{\NN}(\{S_n\times\INT \})\bigr),
$$
such that 
\begin{enumerate}
	\item[(i)] for each {$B\in  \CS \otimes
	\mathscr{B}([0,T])$} with $\EE\,\eta(B) < \infty$,
	$\eta(B):=i_B\circ \eta : \Omega\to \NN
	\, $\footnote{Here, the mapping $i_B$ is defined
	by $i_B:M_{\NN}(\{S_n\})\ni\mu
	\mapsto \mu(B)\in \NN$, $B\in \CS$.}
	is a Poisson random variable with parameter
	$\EE\, \eta(B)$, otherwise $\eta(B)=\infty$ a.s.

	\item[(ii)] $\eta$ is independently 
	scattered, i.e., if the sets
	$B_j \in   \CS\otimes \mathscr{B}([0,T] )$,
	$j=1,\ldots,n$, are disjoint, then the random
	variables $\eta(B_j)$, $j=1,\ldots,n $,
	are mutually independent;

	\item[(iii)] for each $U\in \CS$, 
	the ${\mathbb{N}}$-valued
	process $(N(t,U))_{t\in \INT }$  defined by
	$$
	N(t,U):= \eta(U \times (0,t]), \;\; t\in \INT
	$$
	is $\BF$-adapted, and its increments are stationary
	and independent of the past, i.e., if $t>s\geq 0$, 
	$N(t,U)-N(s,U)=\eta(U \times (s,t])$ is
	independent of $\mathcal{F}_s$.
\end{enumerate}
\end{definition}

\begin{remark}\label{wegensn}
If $\nu$ is a finite measure on a Polish space $(S, \CS)$, then 
for any $U \in \CS$, $N(t, U)$ is a Poisson-distributed random 
variable with parameter $t\, \nu(U)$. 
In particular, the number of jumps is 
finite. However, if $\nu$ is only $\sigma$-finite 
and $\nu(S_n) \to \infty$ as $n \to \infty$, 
then $N(t, S) = \infty$, meaning there are infinitely 
many jumps within any time interval 
$[t_1, t_2]$, where $t_1 < t_2$.
\end{remark}

In Definition \ref{def-PRM},  the assignment
\begin{equation}\label{eq:intensity0}
	\nu: \CS \ni A \mapsto
	\EE\big[\eta( A\times(0,1))\big]
\end{equation}
defines a uniquely determined measure, called the
\textit{intensity measure} of the 
Poisson random measure $\eta$. Moreover, it 
turns out that the \textit{compensator}
$\gamma$ of $\eta$ is uniquely determined by
$$
\gamma: \CS \times \mathscr{B}([0,T])
\ni (A,I)\mapsto  \nu(A)\times \Leb_{[0,T]}(I),
$$
where $\Leb_{[0,T]}$ denotes the Lebesgue 
measure on $[0,T]\subset \RR$. The difference 
between a time-homogeneous Poisson random measure
$\eta$  and its compensator $\gamma$, i.e.,
$$
\tilde \eta:=\eta-\gamma,
$$
is called the \textit{time-homogeneous
compensated Poisson random measure}.

\begin{remark}\label{poissonskorohod}
Fix  $n\in\NN$. If $S_n$ is a metric space, then the process 
\begin{equation*}
	\Omega\times [0,T] \ni (\omega,t)
	\mapsto \eta(\cdot \times [0,t])(\omega)
	\in M_{\NN}(S_n)
\end{equation*}
is c\'adl\'ag and belongs to $\DD([0,T];M_{ \NN}(S_n))$.

If $S_n\subset E$, $E$ is a Banach space of type $p$, 
$1\le p\le 2$, and 
$$
\lim_{n\to\infty} \int_{S_n}|z|^p\nu(dz)<\infty,
$$ 
then
\begin{equation*}
	L:\Omega\times [0,T] \ni (\omega,t)
	\mapsto \int_S z 
	\,\tilde \eta(\cdot \times [0,t])(\omega)\in E
\end{equation*}
is c\'adl\'ag and belongs to $\DD([0,T];E)$.		
\end{remark}

Poisson random measures arise in
a natural way by means of a L\'evy process.

\begin{definition} Let $E$ be a Banach space.
A stochastic process $\{ L(t):t\ge 0\}$ is a L\'evy process
if the following conditions are met:
\begin{itemize}
	\item for any $l\in\NN$ and $0\le t_0<t_1<\cdots <t_{{l}}$,
	the random variables $L(t_0)$, $L(t_1)-L(t_0)$,
	$\ldots$, $L(t_{{l}})-L(t_{{l}-1})$ are independent;
	
	\item $L_0=0$ a.s.;
	
	\item For all $0\le s<t$, the distribution of $L(t+s)-L(s)$
	does not depend on $s$;
	
	\item $L$ is stochastically continuous;
	
	\item the trajectories of $L$ are a.s.\ c\'adl\'ag on $E$.
\end{itemize}
\end{definition}

The characteristic function of a L{\'e}vy process is
uniquely determined by the  L{\'e}vy-Khinchin formula.
Before describing this formula, let us
first introduce the concept of the L{\'e}vy measure $\nu$.

\begin{definition}[{\cite[Chapter 5.4]{Linde:1986mz}}]\label{def:levy}
Let $E$ be a separable Banach space with dual $E^\prime$.
A symmetric $\sigma$-finite Borel measure
$\lambda$ on $E$ is called a \textit{L\'evy measure} if and only if
(i) $\lambda(\{0\} )=0$  and (ii) the
function\footnote{As remarked in \cite[Chapter 5.4]{Linde:1986mz},
we do not need to suppose that $\int_E(\cos\langle x,a\rangle -1)
\,\lambda(dx)$ is finite. However, if $\lambda$ is a symmetric L\'evy
measure, then, for each $a \in E^\prime$, the
integral is finite, see Corollary 5.4.2 in \cite{Linde:1986mz}.}
$$
E^\prime \ni a\mapsto
\exp \left(\int_E \bigl(\cos\langle x,a\rangle -1\bigr)
\, \lambda(dx)\right)
$$
is a characteristic function of a Radon measure 
on $E$ \cite[p. 17]{Linde:1986mz}.

In the measure space $(E, \mathscr{B}(E))$, a $\sigma$-finite
Borel measure $\lambda$ is termed a L{\'e}vy measure if its symmetric
part $\frac{1}{2}(\lambda + \lambda^-)$, where
$\lambda^-(A) = \lambda(-A)$ for $A \in \mathscr{B}(E)$, 
qualifies as a L{\'e}vy measure. The collection of all 
L{\'e}vy measures on $(E, \mathscr{B}(E))$ 
is denoted by $\CL(E)$.
\end{definition}

For the definition of Banach 
spaces of (Rademacher) type $p$, where $p \in [1,2]$, 
see \cite[p.~54]{Hytonen:2017aa} 
and \cite[p.~40]{Linde:1986mz}. Common examples of 
such Banach spaces include $L^q$ 
spaces with $p \leq q < \infty$, defined over bounded domains, as 
well as the corresponding Besov spaces (see \cite{Brzezniak:1999aa} 
for further details). Moreover, if a Banach space $E$ is of type $p$, 
then it is also of type $q$ for all $q \geq p$.

The L\'evy-Khintchine formula establishes that for any $E$-valued
L{\'e}vy process $\{L(t) : t \ge 0\}$, there exist a positive
operator $Q: E' \rightarrow E$, a non-negative
measure $\nu$ concentrated on $E \setminus \{0\}$ with
the property that $\int_E 1 \wedge |z|_E^p \,\nu(dz) < \infty$,
and an element $m \in E$ such that
(see \cite{Applebaum:2009uq,Araujo:1978ab}
or \cite[Theorem 5.7.3]{Linde:1986mz})
\begin{equation}\label{eq:intensity-measure}
	\begin{split}
		\EE\left[ e^{i\la L(1),x\ra}\right]
		& = \exp\Biggl(i\la m,x\ra -\frac 12 \la Qx,x\ra
		\\ & \qquad\qquad
		{-}\int_E \left( 1-e ^ {i\la y,x\ra }+\mathbf{1}_{(-1,1)}(|y|_E)
		i \la y,x\ra \right) \,\nu(dy)\Biggr),
	\end{split}
\end{equation}
for each $x\in E'$. 
We refer to the measure $\nu$ as the \textit{L{\'e}vy measure} 
of the L{\'e}vy process $L$. Furthermore, the triplet $(Q, m, \nu)$ 
uniquely characterizes the law of $L$.

We now proceed to construct a Poisson 
random measure whose intensity measure is determined 
by a L{\'e}vy measure $\nu$. Let $\MA$ be a 
filtered probability space, see \eqref{probabilitysp}, 
and let $E$ be a $p$-stable Banach space for some $p \in [1, 2]$. 
Consider an $E$-valued L{\'e}vy process 
$\{L(t) : t \geq 0\}$, defined on $\MA$, which is of pure 
jump type\footnote{A L{\'e}vy process is of 
pure jump type if $Q = 0$.} with L{\'e}vy measure 
$\nu$. To this process, we associate a counting 
measure $\eta_L$, also defined on $\MA$, as follows:
$$
\mathscr{B}(E)\times 
\mathscr{B}((0,T])\ni (B,I) \mapsto \eta_L(B\times I)
:= \# \bigl\{s\in I \mid \Delta_s L \in B\bigr\} \in \NN.
$$
Here, the jump process $\Delta L 
= \{\Delta_t L : 0\le t<\infty\}$
linked to $L$ is defined by
$$
\Delta_t L := L(t) - L(t-) = L(t) - \lim_{\epsilon\to 0}
L(t-\varepsilon), \quad t> 0, \quad \Delta_0 L=0.
$$
If $\nu$ is symmetric and supported within the unit ball,
then $\eta_L$ function as a time-homogeneous
Poisson random measure with intensity measure $\nu$.
Furthermore, we can express $L(t)$ through the integral
$$
L(t) = \int_0^t \int_Z z\, 
\tilde{\eta}_L(dz,ds), \quad t\ge 0,
$$
where $\tilde{\eta}_L$ denotes the
compensated version of $\eta_L$.

Conversely, when given a time-homogeneous Poisson
random measure $\eta$ on a $p$-stable Banach space $E$,
with $p\in [1,2]$, it is possible
to construct a L\'evy process $L$.
The integral
$$
\int_I\int_Z  z\,\tilde \eta(dz,ds)
$$
is well-defined for any $I\in\mathscr{B}([0,T])$
if and only if the intensity measure $\nu$ 
of $\eta$ is a L\'evy measure, see
\cite[p. 123, Theorem (2.1)]{Dettweiler:1977aa}.

For further information about the connection between Poisson
random measures and L{\'e}vy processes, we direct the reader
to  Applebaum \cite{Applebaum:2009uq}, Ikeda and Watanabe
\cite{Ikeda:1981aa}, and Peszat and Zabczyk \cite{Peszat:2007pd}.

The formulation of stochastic integrals within arbitrary
Banach spaces present notable challenges depending 
significantly on the geometric structure of 
the space \cite{Neerven:2015aa}.
Our discussion in what follows is therefore 
confined to UMD Banach spaces of type $p\in [1,2]$ (and thus 
of martingale type $p$), where the $p$ depends 
on the integrability properties 
of the specific L{\'e}vy measure being used.
Standard examples of Banach spaces with 
martingale type $p\in [1,2]$ 
include Hilbert spaces (martingale type 2), 
$L^p$ spaces, and uniformly convex 
spaces (which have martingale type 2). 
We refer to \cite[Chapters 3.5 \& 4]{Hytonen:2016aa} 
for details. 

We define the space of 
possible integrands as follows:
\begin{align*}
	& \mathcal{M}^p([0,T];L^p(Z, \nu; E))
	:= \Biggl\{ \xi : [0,T] \times \Omega \to L^p(Z,\nu;E) 
	\mid \\ & \qquad  
	\text{$\xi$ is progressively measurable and} \,\,
	\EE \int_0^T \abs{\xi(t)}_{L^p(Z, \nu; E)}^p 
	\, dt < \infty \Biggr\}.
\end{align*}
In \cite{Brzezniak:2009aa}, the 
second author and Brze\'{z}niak
demonstrated the existence and
uniqueness of a continuous linear operator
that assigns to each progressively measurable process
$\xi \in \mathcal{M}^p([0,T];L^p(Z, \nu; E))$ 
an adapted c\'adl\'ag $E$-valued process. 
This stochastic integral process is denoted by
$$
\int_0^\cdot \int_Z \xi(r,x) 
\, \tilde{\eta}(dx,dr),
$$
and it satisfies the following: for
a random step process $\xi \in \CM([0,T], L^p(Z,\nu;E))$ with the form
$$
\xi(r{,z}) = \sum_{j=1}^{l} 
\mathbf{1}_{(t_{j-1},t_{j}]}(r)\, \xi_j{(z)}, 
\quad r\ge 0,
$$
where $\{0=t_0<t_1<\ldots<t_{{l}}\}$ 
is a finite partition of $[0,T]$ and $\xi_j{(z)}$ 
is an $E$-valued, $\CF_{t_{j-1}}$-measurable, 
$p$-summable random variable for each $j$, then
$$
\int_0^t \int_Z \xi(r,z)\, \tilde{\eta}(dz,dr)
= \sum_{j=1}^{{l}}  \int_Z  \xi_j (z)\, \tilde{\eta}
\left(dz,(t_{j-1}\wedge t, t_{j}\wedge t] \right).
$$
Furthermore, the operator
$$
I: \mathcal{M}^p ([0,T];L^p(Z,\nu;E)) \ni\xi
\mapsto \int_0^\cdot \int_Z \xi(r,x)\,\tilde{\eta}(dx,dr)
\in \mathcal{M}^p ([0,T];E)
$$
is continuous. 
That is, there exists a constant $C=C(E)$,
independent of $\xi$, $\eta$, and $\nu$, such that
$$
\EE \left| \int_0^t \int_Z \xi(r,z)
\, \tilde{\eta}(dz,dr) \right|^p \leq C
\EE \int_0^t\int_Z|\xi(r,z)|^p\, \nu(dz)\,dr, \quad t\geq 0.
$$

When working with a Poisson random measure, 
the associated solution process can be c\'adl\'ag and 
predictable in one space while being only 
progressively measurable in another. 
This is why we make a distinction between 
c\'adl\'ag behavior and predictability versus 
progressive measurability. For completeness, we 
provide a simple example to illustrate this point.

\begin{example}\label{stex}
First, let us construct a space-time Poisson random 
measure with a prescribed intensity measure 
$\nu$ on $\RR^d$ (compare also with 
\cite[Proposition 7.21]{Peszat:2007pd}). 
To this end, let us specify $\nu_0$ 
as a L\'evy measure on 
$\RR\setminus \{0\}$, assuming that for 
any $n\in\NN$, the measure  $\nu_0$ is finite 
on $\RR\setminus \left[-\tfrac{1}{n},\tfrac{1}{n}\right]$ 
and $\int_{[-1,1]\setminus\{0\}} |z|^2\, \nu_0(dz)<\infty$.
Set 
$$
S_n:=\left(\RR\setminus \left[ -\tfrac 1n,\tfrac 1n\right]\right)
\times \RR^d
$$
and let us define a measure $\nu_n$ on $S_n$ by
$$
\nu_n(A\times B):=\nu_0(A) \Leb(B),
\quad A\in\mathscr{B}
\left(\RR\setminus \left[ -\tfrac 1n,\tfrac 1n\right]\right),
\,\, B\in \mathscr{B}(\RR^d),
$$
where $\Leb$ denotes the Lebesgue 
measure on $\RR^d$. To ensure that $\nu_n$ is well-defined as a 
measure on $S_n$, we note that the collection 
of rectangular sets $A\times B$, where 
$A \in \mathscr{B}\left(\RR \setminus 
\left[-\frac{1}{n}, \frac{1}{n}\right]\right)$ 
and $B \in \mathscr{B}(\RR^d)$, forms a $\pi$-system. 
Since $\nu_n$ is finitely additive and satisfies the 
properties of a Dynkin system, it extends uniquely 
to the $\sigma$-algebra $\mathscr{B}(S_n)$. 
Introduce the $\sigma$-algebra 
$\mathcal{S}:=\vee_{n\in\NN}
\mathscr{B}(S_n)$.\footnote{Let $Z$ be a set, 
and let $\mathcal{A}$ and $\mathcal{B}$ be collections 
of subsets of $Z$. We define $\mathcal{A} \vee \mathcal{B}$ 
as the smallest $\sigma$-algebra on $X$ that contains 
every set in both $\mathcal{A}$ and $\mathcal{B}$,  
that is, $\mathcal{A} \vee \mathcal{B}$ is generated by the union 
$\mathcal{A} \cup \mathcal{B}$: $\mathcal{A} \vee \mathcal{B} 
= \sigma(\mathcal{A} \cup \mathcal{B})$.} 
We define the measure $\nu$ 
on $\mathcal{S}$ as the unique extension 
of the sequence of measures $\nu_n$ defined on 
$\mathscr{B}(S_n)$ for $n \in \mathbb{N}$. 
This extension is well-defined since any set in 
$\mathcal{S}$ can be approximated 
from below by its restrictions to the 
truncated domains $S_n$.

The space-time Poisson noise on $(\RR^d, \CB(\RR^d))$ with 
jump intensity measure $\nu$, defined over a probability 
space $(\Omega, \CF, \PP)$, is a 
$\CF / \CM(M_\NN(\{S_n \times [0, T]\}))$-measurable 
mapping\footnote{Let $(X_i, \mathcal{X}_i)$, $i = 1, 2$, 
be two measurable spaces. A function $f: X_1 \to X_2$ 
is $\mathcal{X}_1/\mathcal{X}_2$-measurable if, for all 
$A \in \mathcal{X}_2$, the pre-image $f^{-1}(A) := 
\{x \in X_1 : f(x) \in A\}$ belongs to $\mathcal{X}_1$.}
\begin{equation}\label{eq:eta-def}
	\eta : \Omega \to 
	\CM\left(M_\NN(\{S_n \times [0,T]\})\right),
\end{equation}
such that:
\begin{itemize}
	\item For any $U \in \mathcal{S} \otimes \mathscr{B}([0,T])$ 
	with $(\nu \times \Leb_{[0,T]})(U) < \infty$, the 
	random variable $\eta(U) := i_U \circ \eta$ is 
	Poisson-distributed with parameter 
	$(\nu \otimes \Leb_{[0,T]})(U)$.

	\item If $U_1 \in \mathcal{S} \otimes \mathscr{B}([0,T])$ 
	and $U_2 \in \mathcal{S} 
	\otimes \mathscr{B}([0,T])$ are disjoint, 
	then the random variables $\eta(U_1)$ and $\eta(U_2)$ are 
	independent, and $\eta(U_1 \cup U_2) = \eta(U_1) 
	+ \eta(U_2)$ almost surely.
\end{itemize}

The space-time Poisson random measure can 
be represented as a function-valued Poisson 
process with a Poisson random measure on the space 
$E_0 := B_{2,\infty}^{-\frac{d}{2}}(\RR^d)$.\footnote{The choice 
of $E_0$ is motivated by the fact that the Dirac 
measure $\delta$ belongs to the Besov space 
$B_{2,\infty}^{-\frac{d}{2}}(\RR^d)$ ($=E_0$) 
(see \cite[p.~34]{Runst:1996aa}). 
Since the Besov embedding $B_{2,2}^{-\frac{d}{2}}(\RR^d) 
\hookrightarrow B_{2,\infty}^{-\frac{d}{2}}(\RR^d)$ holds 
(see \cite[p.~30]{Runst:1996aa}), the Dirac measure $\delta$ 
does not belong to the Sobolev space $H_2^{-\frac{d}{2}}(\RR^d)$, 
which coincides with the Lizorkin-Triebel 
space $F_{2,2}^{-\frac{d}{2}}(\RR^d)$ 
and the Besov space $B_{2,2}^{-\frac{d}{2}}(\RR^d)$. 
However, $\delta$ belongs to the Sobolev space 
$H_2^\gamma(\RR^d)$ for $\gamma < -\frac{d}{2}$.} 
To verify this claim, let us first define the mapping
$$
f_n:S_n\longrightarrow E_0:
(z,x)\mapsto f_n(z,x)=z\,\delta_x,
$$
which maps points $(z, x) \in S_n$ 
to elements of the function space $E_0$.
The set
$$
E_n := \bigl\{ f_n(z,x) \mid (z,x)\in S_n\bigr\}
\subset E_0,
$$
represents the image of $S_n$ under $f_n$, 
for each $n\in \NN$.  
Since $E_0$ is not of type $2$, we embed 
$E_0$ into the (type $2$) Besov space 
$E:=B_{2,2}^{-\gamma }(\RR^d)$, 
where $\gamma>\frac{d}{2}$ is arbitrary. 
For each $n\in\NN$, the measure $\mu_n$ on $E$ is then 
defined as the pushforward of the measure $\nu_n$ under $f_n$, i.e.,
$$
\mu_n(B):= \nu_n(f_n^{-1}(B\cap E_n)), 
\quad B\in\mathscr{B}(E).
$$
It is straightforward to verify that $\mu_n$ 
converges to a L\'{e}vy measure on $E$ as $n \to \infty$. 
With a slight abuse of notation, we will 
denote the limit measure by $\nu$.

Now, let us consider the solutions $\xi$ 
and $\xi_n$, $n\in \NN$, to the following SPDEs:
$$
(\ast)\,\,\, 
d\xi(t)-\Delta \xi(t)\, dt 
= \int_E z \, \tilde \eta(dz, dt),
\quad
(\ast)_n \,\,\,
d\xi_n(t) - \Delta \xi_n(t)\, dt 
= \int_{E_n} z \,\tilde  \eta(dz, dt),
$$
where $\tilde \eta$ is the compensator of the 
random measure $\eta$ defined in \eqref{eq:eta-def}, 
and we use additive noise for simplicity of presentation. 
It follows that
$$
[0, T] \ni t \mapsto \xi(t) = \int_0^t \int_E 
e^{-(t-s)\Delta} z \, \tilde \eta(dz, ds)
$$
and
$$
[0, T] \ni t \mapsto \xi_n(t) = \int_0^t \int_{E_n} 
e^{-(t-s)\Delta} z \, \tilde \eta(dz, ds),
$$
are the unique solutions to $(\ast)$ and $(\ast)_n$, respectively, 
where $(e^{\Delta t})_{t \geq 0}$ denotes the heat semigroup. 
Consequently, the processes $\xi$  and $\{\xi_n: n\in\NN\}$ 
are c\'adl\'ag in the ``large space" $E=B_{2,2}^{-\gamma }(\RR^d)$.  
More precisely, the process $\xi_n$ is c\'adl\'ag 
in $E$, and for all $t \in [0\, T]$,
$\lim_{s \uparrow t} \xi_n(s)$ is predictable. Moreover, the 
process $[0, T] \ni t \mapsto \lim_{s \uparrow t} \xi(s)$ 
is also c\'adl\'ag and predictable in this space.

In the ``smaller space" setting 
of $E_0= B_{2,\infty}^{-\frac{d}{2}}(\RR^d)$, 
the processes $\xi_n$ remain  c\'adl\'ag, due to 
the finiteness of the L\'evy measure. In particular, $\PP$-a.s., 
$\xi_n \in \mathbb{D}([0, T]; E_0)$. In addition, 
for $\tilde E:=B_{2,2}^{-\frac{d}{2}}(\RR^d)\subset E$,
$$
\sup_{0 \leq t \leq T} \EE \abs{\xi_n(t)}_{\tilde E}^2 < \infty,
\quad n\in \NN.
$$
However, $\xi$ is not càdlàg in $E_0$, as the space is 
not of type $p$. Nevertheless, we have
$$
\EE \int_0^T \abs{\xi(s)}_{\tilde E}^2 \, ds < \infty,
$$
and $\xi$ is progressively measurable in $\tilde E$. 
In particular, for all $s \in \left[-\frac{d}{2},\frac{d}{2}+1\right)$, 
there exists a constant $C > 0$, 
independent of $n\in \NN$, such that
$$
\EE \int_0^T \abs{\xi_n(s)}^2_{B_{2,2}^s}\, ds\le C,
\quad 
\EE \int_0^T \abs{\xi(s)}_{B_{2,2}^s}^2\, ds\le C.
$$

In closing, let us mention that to establish the 
progressive measurability of a specific version 
of $\xi_n$, where $n \in \NN$, one may employ 
the $k$-th order shifted Haar projection, denoted as $\xi_n^{(k)}$. 
This approach use piecewise constant approximations of $\xi_n$ 
to construct a sequence $(\xi_n^{(k)})_{k \in \NN}$ that is 
predictable. Moreover, one can 
show that as $k \to \infty$,
$$
\xi_n^{(k)} \to \xi_n \quad 
\text{in $L^2(\Omega; L^2([0, T],\tilde E))$}.
$$
For a proof, see 
\cite[Appendix C]{Brzezniak:2011aa}.
\end{example}


\section{Kurtz's framework}\label{kframework}
Our goal is to extend the Yamada-Watanabe theorem, which 
links the existence and uniqueness of weak and strong solutions 
to stochastic equations, to a broader class of L{\'e}vy-driven SPDEs 
using Kurtz’s abstract principles \cite{Kurtz:2014aa,Kurtz:2007aa}. 
Following \cite{Liu:2015aa} and \cite{Brzezniak:2014aa}, we revisit 
the variational SPDE framework and demonstrate how filtration 
and regularity conditions can be 
incorporated into Kurtz’s approach.

Let us turn our attention to the abstract 
framework introduced in \cite{Kurtz:2014aa} 
and \cite{Kurtz:2007aa}. Consider $B_1$ and $B_2$, both 
metric spaces, and a Borel measurable function
$\Gamma: B_1 \times B_2 \to \RR$. Let there be a random
variable $Y$ (the input of the model), 
taking values in $B_2$, with law $\rho$,
and being defined on a given probability 
space $(\Omega, \CF, \PP)$. Our focus is on 
identifying a solution to the equation
$\Gamma(X,Y) = 0$. Specifically, we seek 
a random variable $X$, taking values 
in $B_1$ and being defined over $(\Omega, \CF, \PP)$,
that satisfies
\begin{equation}\label{sol01}
	\Gamma(X,Y) = 0
	\quad
	\text{in the sense that}
	\quad
	\PP\left( \bigl\{ \Gamma(X,Y) = 0\bigr\}\right) =1.
\end{equation}

We introduce the concept of a strong solution 
by defining it as follows:

\begin{definition}\label{strongsol}
A pair $(X,Y)$ constitutes a strong solution to the 
problem \eqref{sol01} defined by $(\Gamma,\rho)$, 
where $\Law(Y)=\rho$, if there exists a
Borel measurable function $F:B_1 \to B_2$ 
such that \eqref{sol01} holds with 
$X=F(Y)$, $\PP$-a.s.
\end{definition}

If a strong solution exists over  some probability space 
$(\Omega, \CF, \PP)$, then $(X,Y)$ has  a  joint distribution. 
In particular, there exists a probability measure
$$
\distr:\mathscr{B}(B_1\times B_2)\to[0,1],
$$
such that {$\distr\left (B_1\times A\right)=\rho(A)$} 
for all $A\in\mathscr{B}(B_2)$ and
$\PP(X\in A)=\int_{B_2} \distr(A,y)\rho(dy)$ 
for all $A\in\mathscr{B}(B_1)$.
Let us denote this joint distribution of 
the two random variables $X,Y$ by $\mu_{X,Y}$. 
If $(X,Y)$ is a solution to \eqref{sol01}, then 
$\distr_{X,Y}$ is determined by the distribution $\rho$ of 
the input $Y$ and the mapping $F$, 
see \cite[Lemma 1.3]{Kurtz:2014aa}.

To further elaborate on the problem from the perspective
of probability laws, we are interested in finding 
a joint distribution---also called 
a joint solution measure---$\mu \in \mathcal{P}(B_1 \times B_2)$.
This probability measure $\mu$ should satisfy 
$\mu(B_1 \times A) = \rho(A)$ for all measurable
subsets $A \in \mathscr{B}(B_2)$, along
with the requirement that:
\begin{equation}\label{solma}
	\int_{B_1 \times B_2}
	\left|\Gamma(x, y)\right|
	\,\mu(dx, dy) = 0.
\end{equation}
Given $\Gamma$ and $\rho$,  we define 
$\mathcal{S}_{\Gamma,\rho}$ as follows:
\begin{equation}\label{eq:S-Gamma-rho}
	\begin{split}
		& \text{$\CS_{\Gamma,\rho}$ is the 
		set of all joint solution measures 
		$\mu \in \mathcal{P}(B_1 \times B_2)$}
		\\ & 
		\text{for which the constraint 
		\eqref{solma} is satisfied 
		and $\mu(B_1 \times \cdot) = \rho$.}
	\end{split}
\end{equation}
In this way, we can speak also of a weak solution 
to the problem $(\rho, \Gamma)$:

\begin{definition}\label{weaksol}
A weak solution of the problem $(\Gamma,\rho)$
is a pair of random variables $(X,Y)$ 
defined on a probability 
space $(\Omega,\CF,\PP)$ such that $Y$ has 
distribution $\rho$ and $(X, Y )$ meets the 
constraints in $\Gamma$, that is, 
$\mu_{X,Y}\in \mathcal{S}_{\Gamma,\rho}$, 
see \eqref{eq:S-Gamma-rho}.
\end{definition}

Let us now consider the setting in 
which we encounter SPDEs in the variational framework.
Let $H$ be a separable Hilbert space with inner product
$\langle \cdot,\cdot \rangle$ and let $\hast$ denote its dual.
Consider $V$ as a reflexive Banach space embedded
continuously and densely into $H$. 
Through the Riesz isomorphism, which 
identifies $H$ with $\hast$, we establish the
Gelfand triple 
\begin{equation}\label{eq:Gelfand}
	\text{$V\subset H\equiv H'\subset V'$ 
	as continuous and dense}.
\end{equation}
The duality pairing between $\vast$ 
and $V$, defined as
$$
\mathbin{_{\vast}}\langle z, v \rangle_{V}
:=z(v), \quad 
z \in \vast, \,\, v \in V,
$$
satisfies
$$
\mathbin{_{\vast }}\langle z, v \rangle_{V}
 = \langle z, v \rangle_H, \quad 
z \in H, \,\, v \in V.
$$
In the following, we consider a filtered 
probability space $\MA$ as 
given by \eqref{probabilitysp}. 
Let $\bW$ be a cylindrical Wiener
process on a given Hilbert space $\CH$, expressed as
\begin{equation}\label{eq:wiener_represent}
	\bW(t) = \sum_{{k}=1}^\infty h_k\beta_k(t),\quad t\in[0,T],
\end{equation}
where $\left\{h_k \mid k\in\NN\right\}$ is an orthonormal basis in $\CH$
and $\left\{\beta_{k} \mid  k \in \NN\right\}$ is a sequence of mutually
independent Brownian motions defined over $\MA$.

Let $\eta$ be a time-homogeneous Poisson random
measure, with a given intensity measure $\nu$ defined 
on a given Polish space $(S,\CS)$. 
It is assumed that $\eta$ is independent of
the Wiener process $\bW$ and defined
over the same filtered probability space $\MA$.
See Definition \ref{def-PRM} for further details.

\medskip

Moving forward, we are presented with the 
following measurable mappings:
\begin{itemize}
	\item $b: [0,T] \times V \rightarrow \vast $, 
	mapping into the dual space of $V$;
	
	\item $\sigma: [0,T] \times H \rightarrow
	\mathcal{L}_{\text{HS}}(\CH, \vast )$, which 
	maps into the space of Hilbert-Schmidt 
	operators from $\CH$ to $\vast $;
	
	\item $c: [0,T] \times \{S_n\} \times V_d \rightarrow \vast$, 
	where the sets $\{S_n\}$ are defined 
	in Lemma \ref{measure_lemma}, 
	and $V_d\subset V'$ is dense.
\end{itemize}

Let $E_1$ and $E_2$ be Banach spaces, with 
$E_2 \hookrightarrow E_1$ and $V \hookrightarrow E_1$. 
Suppose $E_1$ is a UMD space of type 2, 
which implies that $E_1$ is also of martingale 
type 2 \cite{Hytonen:2016aa}.\footnote{If one deals 
only with a Poisson random measure, it is also possible 
to consider a UMD Banach space of type $p$, with $p \in [1,2]$, 
provided the small jumps are $p$-integrable. 
However, the stochastic It\^{o} integral with respect 
to a Wiener process is defined only 
on a UMD Banach space of type $2$.}
We consider general SPDEs of the form
\begin{equation} \label{SPDE0}
	\begin{split}
		dU(t) & = b(t,U)\, dt+ \sigma(t,U(t))\, dW(t) 
		+\int_S c(t,z,U(t))\, \tilde \eta(dz,dt),
		\\ U(0)&=U_0\in E_2,
	\end{split}
\end{equation}
where $b$, $\sigma$, and $c$ are the 
``coefficients" given above. We refer to a stochastic process 
$U: \Omega \times [0, T] \to E_1$ 
as a solution of \eqref{SPDE0} if the equation
\begin{equation}\label{SPDE}
	\begin{split}
		& \left \langle U(t),\varphi \right\rangle =
		\left\la U_0,\varphi\right\ra
		+\int_0^{t} \left\langle b(s,U(s)),\varphi \right\rangle\,ds
		\\ & \quad
		+\int_0^{t}\sum_{{k}=1}^\infty 
		\left \langle \sigma(s,U(s))[h_{k}],
		\varphi \right\rangle\, d\beta^{k}(s)
		+\int_0^{t}\int_S
		\left \langle c(s,z,U(s)),\varphi \right \rangle
		\, \tilde\eta(dz,ds),
	\end{split}
\end{equation}
is satisfied $\PP$-a.s., for each $t \in [0, T]$, 
and for each test function $\varphi \in V$.

We now formulate the above SPDE problem  
\eqref{SPDE0}, \eqref{SPDE} in an abstract setting. 
Before proceeding, let us recall that all 
random variables are defined over a filtered 
probability space $\MA$ (see \eqref{probabilitysp}). 
Let $Y$ represent the Wiener process $\bW$, the Poisson 
random measure $\nu$, and the given initial condition, 
such that $Y = (\bW, \nu, U_0)$. Furthermore, we 
introduce the following space linked to $Y$:
\begin{equation}\label{defB2}
	B_2= C_b([0,T]; \WienH) \times 
	M_{\NN}(\{S_n \times \INT\}) 
	\times E_2.
\end{equation}
At the same time, in the abstract formulation under 
consideration, let $X$ denote the solution variable $U$ 
belonging to the Skorohod space of c\'adl\'ag functions 
with values in $E_1$, which is denoted by
\begin{equation}\label{eq:B1-def}
	B_1 = \DD([0,T]; E_1).
\end{equation}

\begin{example}
To demonstrate the applicability of our results, we 
give an example of an SPDE 
that fits within our framework \eqref{SPDE0}, \eqref{SPDE}, 
along with a specification of the relevant spaces. 
Let\footnote{The dual is defined with respect to 
the Hilbert space $H_0^{1,2}(\CO)$.}
$$
V:=L^p(\CO)\subset H:=H_0^{1,2}(\CO)
\subset (L^p(\CO))'=:V'
$$
be a Gelfand triple equipped  
with the following scalar product
$$
{ }_{V^*}\la u,v\ra _V :=\int \la \nabla u(x),\nabla v(x)\ra \, dx 
\quad u,v\in H_0^{1,2}(\CO).
$$
Here, $\CO$ denotes a bounded open 
subset of $\RR^d$. 
Let $b:V\to V'$ be the porous medium 
operator defined by
$$
b(u):=\Delta \left(|u|^{p-2} u\right), 
\quad u\in L^p(\CO), \quad p\ge 2.
$$
Note that $b$ is hemicontinuous, locally 
monotone, coercive, and bounded; 
see, for example, \cite{Barbu:2016aa} 
or \cite[p.~87]{Liu:2015aa}.

Let $\CH = \RR^m$, and assume that 
$\{W_t\}_{t \geq 0}$ is an $m$-dimensional 
Wiener process on the filtered probability 
space \eqref{probabilitysp}, represented as in 
\eqref{eq:wiener_represent}. Additionally, let $\eta$ denote 
the space-time Poisson random measure constructed 
in Example \ref{stex}, with intensity measure 
$\nu_0((a,b))=\int_a^b|z|^{-\alpha}e^{-|z|}\, dz$, 
for $a, b \in \RR$ with $a < b$ and $a, b \neq 0$.

The example is now provided by the following SPDE:
\begin{equation*}
	d U(t) =b(U(t)) \, dt+\sigma (U(t)) \, dW(t) 
	+\int_{S} c(z,U(t)) \, \tilde{\eta}(dz,dt),
\end{equation*}
where $U_0$ is an $\CF_0$-measurable 
random variable and $\sigma(u)[h] := u h$ 
for all $u \in V$ and $h \in \CH$. 
Regarding the jump-noise amplitude $c$, we assume
$$
c: B^{-\frac{d}{2}}_{2,\infty}(\CO) \times H^{1,2}_0(\CO)
\ni (z, x) \mapsto x (I - \Delta)^{-\frac{d}{2}-2}z 
\in H^{1,2}_0(\CO),
$$
which implies that, in the abstract setup \eqref{SPDE0}, the 
sets $\{S_n\}$ have been replaced by 
$B^{-\frac{d}{2}}_{2,\infty}(\CO)$ (see Example \ref{stex}), and $V_d$ 
has been replaced by $H_0^{1,2}(\CO)$. 
As $H_0^{1,2}(\CO)\subset V'$, the noise operator 
$c$ maps into $V'$, satisfying the required condition. 
Moreover, note that the coefficients $b$, $\sigma$, and 
$c$ are assumed here to be independent of $t$.

Finally, to complete the identification with the abstract 
framework \eqref{SPDE0}, let $E_1$ be a Banach space such 
that $V’ \hookrightarrow E_1$ continuously; for example, 
$E_1 = H^{-\frac{d}{2}-1,2}(\CO)$, and define
\begin{equation*}
	B_1 = \DD([0,T]; E_1),
\end{equation*}
Setting $E_2 = H^{1,2}_0(\CO)$ (for example), 
we may define the space $B_2$ as
\begin{equation*}
	B_2 = C_b([0,T]; \RR^m) \times
	M_{\NN}\left(\{\RR \setminus 
	\left[-\tfrac{1}{n},\tfrac{1}{n}\right] 
	\times \CO \times [0,T]\}\right)\times H^{1,2}_0(\CO).
\end{equation*}
\end{example}

Now, one must construct a mapping $\Gamma: B_1 \times B_2 \to \RR$ 
such that the solution $X$ of the abstract equation \eqref{sol01} 
coincides with the solution of the SPDE \eqref{SPDE0}. 
In \eqref{sol01}, the solution is defined as a 
random variable that satisfies a constraint given by the 
mapping $\Gamma : B_1 \times B_2 \to \RR$. 
In our SPDE example, however, the solution is 
defined as a process that satisfies a collection 
of equations or constraints.

To be more precise, let $V_d := 
\left\{\varphi_\kk : \kk \in \NN \right\}$ 
be a dense countable subset of $V$ 
(see \eqref{eq:Gelfand}), and let 
$\QQ_T := \QQ \cap [0, T]$. Then, for each $\kk \in \NN$ 
and $t \in \QQ_T$, we define a constraint 
$\Gamma_{\varphi_\kk, t} : B_1 \times B_2 \to \RR$ 
based on \eqref{SPDE} with $\varphi = \varphi_\kk$. 
Within a given probability space, 
see \eqref{probabilitysp}, we define
\begin{equation}\label{gammaus}
	\Gamma =\left\{\Gamma_{\varphi,t}: 
	\varphi\in V_d, \,\, t\in \QQ_T\right\}
\end{equation}
by
\begin{equation}\label{abstracteqnew}
	\begin{split}
		\Gamma_{\varphi,t} \bigl(U,(\bW,\eta,U_0)\bigr)
		& = \langle U_0,\varphi\rangle
		+ \int_0^{t}\left\langle b(s,U(s)),
		\varphi\right\rangle\,ds
		\\ &\qquad
		+\int_0^{t}\sum_{{k}=1}^\infty
		\left\langle\sigma(s,U(s))[h_{k}],
		\varphi\right \rangle\, d\beta_{k}(s)
		\\ & \qquad
		+\int_0^{t}\int_S
		\left \langle c(s,x,U(s)),
		\varphi \right \rangle
		\, \tilde\eta(dx,ds)
		\\ & \qquad 
		- \langle U(t),\varphi\rangle.
	\end{split}
\end{equation}
Omitting the probability variable, we write 
$\Gamma_{\varphi, t}(x, y)$, where $\Gamma_\varphi: 
B_1 \times B_2 \to \mathbb{R}$ is defined in \eqref{abstracteqnew}, 
and $B_1$ and $B_2$ are defined in \eqref{eq:B1-def} and \eqref{defB2}, 
respectively. Here, $x$ resides in $B_1$, symbolizing the 
variable $U$, and $y$, located in $B_2$, represents the 
input triplet $(\bW, \eta, U_0)$. Whenever we want to 
emphasize the dependence on $t$, we write 
$\Gamma_\varphi(x, y)(t)$ instead 
of $\Gamma_{\varphi, t}(x, y)$.

\begin{remark}\label{sufficent}
We will consider probabilistic weak solutions, also known 
as martingale solutions. Typically, the probability space associated 
with these solutions does not coincide with the probability space 
fixed in \eqref{probabilitysp}; instead, the focus is primarily 
on the solution measure. To formulate the problem, 
it is sufficient to specify:
\begin{itemize}
	\item the Hilbert space $\CH$, where the cylindrical 
	Wiener process is defined,

	\item the Polish space $(S, \CS)$ (along with a 
	sequence $\{S_n\}$ of sets) and the intensity measure $\nu$ 
	defined on $(S, \CS)$, which characterizes the L{\'e}vy process,
	
	\item the distribution $\rho_0$ defined on $E_2$, 
	characterizing the initial condition.
\end{itemize}
Additionally, the coefficients $b$, $\sigma$, and $c$ 
of the SPDE \eqref{SPDE0} are required. 
From $\CH$, $\nu$, and $\rho_0$, one can 
construct a filtered probability space $(\Omega, \CF, \BF, \PP)$ 
with a cylindrical Wiener process $\bW$ on $\CH$, a Poisson random 
measure $\eta$ on $S$, and an initial condition $U_0$ 
(with law $\rho_0$). Here, $U_0$ is 
$\CF_0$-measurable, and $\eta$ 
and $\bW$ are independent, both 
adapted to the filtration 
$(\CF_t)_{t \in [0,T]}$.
\end{remark}

After introducing the problem, we proceed to define the 
concept of a solution as it will be applied in the subsequent 
sections. Occasionally, we need to assume 
additional regularity properties that, 
while not part of the formal 
definition of the solution, are crucial 
for ensuring pathwise uniqueness. 
These regularity properties can be introduced 
through additional mappings:
$$
\theta^{\alpha_i}_i:\DD([0, T]; E_1) \to \RR,
\quad \alpha_i \in A_i, \quad i = 0, 1,
$$
where $A_i$, $i = 0, 1$, are two index sets. 
We define
\begin{equation}\label{calX}
	\begin{split}
		\mathcal{X} := \Bigl\{
		U\in B_1 & \mid  \EE \, \theta_0^{\alpha_0}(U)\le R,
		\,\, 
		\PP\bigl( \theta_1^{\alpha_1}(U)<\infty\bigr)=1, 
		\\ & \qquad 
		\mbox{for all $\alpha_i\in A_i$, $i=0,1$} \Bigr\},
	\end{split}
\end{equation}
for some given $R>0$.
The set $\mathcal{X}$ allows us to 
incorporate this additional information.

\begin{remark}
To demonstrate how additional regularity assumptions 
can be applied to the solution $U$ through the mappings 
$\theta^{\alpha_0}_0$ and $\theta^{\alpha_1}_1$, 
let us consider an example. Specifically, 
define the first functional 
$\theta_0^{\alpha_0}(U)$ as
$$
\theta_0^{\alpha_0}(U)
:=\norm{U}_{L^2(0, T; V)}^2, 
\quad \forall \alpha_0.
$$
Imposing the condition $\EE\,\theta_0^{\alpha_0}(U) \leq R$, 
for some constant $R$, encodes a boundedness 
constraint on $U$. This approach facilitates the enforcement 
of regularity and additional bounds on the solutions, 
extending beyond the requirements of the solution concept 
but necessary for a well-posedness analysis.

Furthermore, non-negativity constraints can be incorporated 
using the second functional by defining, e.g.,
$$
\theta_1^{\alpha_1}(U(t,x))
:=
\begin{cases}
	\infty & \mbox{ if $U(t,x)<0$},
	\\ 
	0 & \mbox{elsewhere}.
\end{cases}
$$
We then require that the probability of the event 
$\theta_1^{\alpha_1}(U) <\infty $ is equal to one.
In this manner, if, for instance, the solution space 
$E_1 = L^2(\CO)$ in \eqref{SPDE0} is used, 
where $\CO\subset \RR^d$ is bounded 
open, the set in \eqref{calX} 
transforms into the measurable set:
\begin{align*}
	\mathcal{X} := 
	\Bigl\{ U & :[0,T]\to L^2(\CO) \mid 
	\EE\norm{U}_{L^2(0,T;L^2(\CO))}^2\le R,
	\\ & \qquad 
	\PP\left(\left\{\Leb\left(
	\bigl\{(t,x):U(t,x)< 0\bigr\}\right)
	=0\right\}\right)=1
	\Bigr\}.
\end{align*}
\end{remark}

We now present the precise definition of a solution within Kurtz’s framework. 
We consider a Gelfand triple $(V,H,V')$, see \eqref{eq:Gelfand}, 
along with Banach spaces $E_1$ and $E_2$. Here, $E_2$ is 
continuously embedded into $E_1$, and $V$ is 
continuously embedded into $E_1$.  
Additionally, $E_1$ is a UMD space of type 2.

\begin{definition}\label{def_solution}
Given a Hilbert space $\CH$, an
intensity measure $\nu$ over a Polish space $(S,\CS)$, 
see \eqref{eq:intensity0} and \eqref{eq:intensity-measure}, 
and a distribution $\rho_0$ on $E_2$, we consider 
a tuple $(\MA, U, \bW, \eta,U_0)$ to 
be a probabilistic weak  solution of the 
SPDE \eqref{SPDE0}, under the following conditions: 
The tuple consists of:
\begin{enumerate}
	\item[(i)] A filtered probability space $\MA = (\Omega, \CF, \BF, \PP)$,
	where $\BF = (\CF_t)_{t \in \INT}$ denotes the filtration.
	
	\item[(ii)] A cylindrical Wiener process $\bW$ on $\WienH$,
	defined over $\MA$ and adhering to the
	representation \eqref{eq:wiener_represent}.
	
	\item[(iii)] A time-homogeneous Poisson random measure
	$\eta$ on $(S, \CS)$, with intensity 
	measure $\nu$, defined over $\MA$.
	
	\item[(iv)] An initial condition $U_0$, which 
	is an $E_2$-valued random variable over $\MA$ 
	(with law $\rho_0$) and is $\CF_0$-measurable.
	
	\item[(v)] A process $U$ on $[0, T]$, which is $\BF$-progressively 
	measurable in $H$ and exhibits c\`adl\`ag behavior in $E_1$.
\end{enumerate}
This setup satisfies the following conditions:
\begin{enumerate}
	\item[(vi)] For all $t \in [0,T]$, $\PP(U(t) \in E_1) = 1$,
	and $U \in \mathcal{X}$---see \eqref{calX}.

	\item[(vii)] The integrals
	\begin{equation}\label{finiteintegrals}
		\begin{split}
			& \int_0^{t} \left|\left\langle b(s,U(s)),\varphi
			\right\rangle\right|\,ds
			+\int_0^{t} \sum_{{k}=1}^\infty
			\left| \left\langle \sigma(s,U(s))[h_{k}],
			\varphi\right \rangle\right|^2\, ds
			\\ & \quad
			+ \int_0^{t}\int_{\{x \in S \mid
			 \left|\left\langle c(s,x,U(s)),
			\varphi \right\rangle\right|_{\vast } < 1\}}
			\left|\left\langle c(s,x,U(s)),
			\varphi\right\rangle \right|^p \,\nu(dx)\,ds
			\\ & \quad
			+\int_0^{t}\int_{\{x \in S \mid \left|\left\langle {c(s,x,U(s))},
			\varphi \right \rangle \right|_{\vast } \ge 1\}}
			\left| \left\langle {c(s,x,U(s))},
			\varphi \right\rangle \right| \,\nu(dx)\,ds
		\end{split}
	\end{equation}
	\vspace*{0.1cm}
	
	\noindent are finite, $\PP$-a.s., for every $t \in [0,T]$
	and every $\varphi \in V$.
	
	\item[(viii)] The process $U$ satisfies 
	\eqref{SPDE}, $\PP$-a.s., $\forall t \in [0,T]$ 
	and $\forall \varphi \in V$.
\end{enumerate}
\end{definition}

The typical well-posedness approach begins with establishing 
the existence of a martingale solution. 
Once existence is ensured, the focus shifts 
to proving the uniqueness of the solution. 
A key challenge in this step is that the concept of solution 
depends on the chosen definition of the stochastic integral. 
The choice of integral---whether It{\^o}, 
Stratonovich, or Marcus---can result 
in different solution processes, depending on the 
noise coefficients $\sigma$ and $c$. 
Therefore, before addressing 
uniqueness, one must specify the type 
of stochastic integral being used.
In the context of It{\^o} calculus, especially dealing 
with L\'evy processes,  concepts such as 
adaptivity, progressively measurability, and predictability 
must be considered to properly define 
the solution process. For details, see, e.g., 
\cite[Chapter]{Revuz:1999wi}.

In the framework of the Itô integral, an essential condition 
is that the increments $\bW(s) - \bW(t)$, 
where $s > t \geq 0$, are independent of the solution process 
$U(t)$. The same requirement applies to the increments of the 
Poisson process. To address this, the following 
definition is crucial:

\begin{definition}\label{def:information}
Let $\bW\in C_b([0,\infty);\WienH)$ and
$\eta\in{M_{{\Bbb N}}(\{S_n \times\INT\})}$  be
the Wiener process and the Poisson random
measure introduced before. Then we define
\begin{equation}\label{eq:wt}
	\begin{split}
		&\mathscr{W}_t(h) =
		\sigma\left(\left\{\left\langle  \bW(s),
		h\right\rangle : 0\le s\le t\right\}\right),
		\quad h \in \CH,
		\\ &
		\mathscr{W}^t(h)=\sigma\left(\left\{
		\left\langle \bW(s)-\bW(t),
		h\right \rangle: t\le s\le T\right\}\right),
		\quad h \in \CH,
	\end{split}		
\end{equation}
and
\begin{equation}\label{eq:etat}
	\begin{split}
		\eta_t(V) &=\eta(V\cap(S\times[0,t] )),
		\quad V\in \mathcal{S}\otimes\mathscr{B}(\INT).
		\\
		\eta^t(V) & =\eta(V\cap (S\times (t,T])),
		\quad V\in \mathcal{S}\otimes\mathscr{B}(\INT ).
	\end{split}
\end{equation}
\end{definition}
The proof of the next lemma is straightforward 
and will, therefore, be omitted.

\begin{lemma}\label{adapttimhom}
If $\bW$ is a Wiener process and $\eta$
is a time-homogeneous Poisson random measure
over a filtered probability space $\MA$ \eqref{probabilitysp}. 
Then, for every $t\in \INT $, $\mathscr{W}_t(h)$ 
and $\eta_t(V)$ are $\mathcal F_t$-measurable
random variables. In addition, $\mathscr{W}^t(h)$ and
$\eta^t(V)$ are independent of $\mathcal F_t$.
\end{lemma}

In the context of SPDEs, incorporating time into the 
Kurtz framework requires extending the probability space 
framework to include a filtration. 
We work on a filtered probability space 
$\MA=(\Omega, \CF, \BF, \PP)$, where 
$\BF = (\CF_t)_{t \in [0, T]}$ 
represents the filtration (see \eqref{probabilitysp}). 
Specifically, the solution process is 
typically progressively measurable 
with respect to the filtration generated 
by both the Wiener process 
and the Poisson random measure. 
To handle this additional complexity, Kurtz introduced 
the concept of \textit{compatibility}. Before defining this concept, however, 
it is necessary to present some additional definitions.

In our setting, the initial condition $U_0$ and 
the processes $\bW$ and $\eta$ are given, 
where $\bW$ is $C_b([0,T];\mathcal{H})$-valued and $\eta$ 
is $M(\{S_n\times [0,T]\})$-valued. The initial condition $U_0$ is 
assumed to be $\CF_0$-measurable. The processes $\bW$ 
and $\eta$ naturally generate a filtration 
$(\CF^Y_t)_{t \in [0,T]}$, where $Y = (\bW, \eta,U_0)$, on the 
underlying probability space $\Omega$.  
In the following definition, we introduce what 
is known as the induced filtration on 
$\mathbb{D}([0,T];E)$, where $E$ is a Banach space.

\begin{definition}\label{def38}
Let $Z$ be a $\DD([0,T];E)$-valued 
random variable on a probability space 
$(\Omega,\mathcal{F}, \PP)$. 
Denote by $Z_t: \DD([0,T]; E) \to  E$ the 
evaluation map $Z\mapsto Z_t$. For $t \in [0,T]$, 
define $\mathcal{B}_t^Z = \sigma(\{Z_s:s \leq t\})$ 
as the $\sigma$-algebra on $\DD([0,T];E)$
generated by the values of $Z$ up to time $t$.
Now, let us introduce the $\sigma$-algebra of 
the preimages of $\mathcal{B}_t^Z$ on $\Omega$. 
For any $t \in [0, T]$, let $\SF_t^{Z}$ be the 
coarsest $\sigma$-algebra with respect 
to which the mapping
$$
Z:(\Omega,\SF_t^{Z}) \to 
(Z_t,\mathcal{B}_t^Z)
$$
is measurable. We refer to the filtration 
$(\SF_t^{Z})_{t \in [0, T]}$ 
as the filtration induced by the random variable $Z$ on 
the probability space $(\Omega, \mathcal{F}, \PP)$.
\end{definition}

\begin{remark}
A $\sigma$-algebra is said 
to be generated by a family of sets if the family of sets 
and the $\sigma$-algebra are part of the same Borel $\sigma$-algebra. 
In the definition above, $\{Z_s : s \leq t\}$ and $\mathcal{B}_t^Z$ 
belong to $\mathscr{B}(\DD([0,T];E))$.

Let $X_1$ and $X_2$ be Banach spaces, 
and let $f: X_1 \to X_2$ be a map.  A $\sigma$-algebra 
is said to be induced by $f$ if the family of preimages 
$\{f^{-1}(A) : A \in \mathscr{B}(X_2)\}$ generates 
a $\sigma$-algebra on $X_1$. Here, $f^{-1}(A)$ denotes 
the set $\{x \in X_1 : f(x) \in A\}$. 
In the context of Definition \ref{def38}, the random 
variable $Z$ serves as the mapping $f$.
\end{remark}

Define $Y = (\bW,\eta,U_0)$, and let 
$(\SF^{Y}_t)_{t \in [0, T]}$ be the filtration 
induced by $Y$ on the underlying probability space $\Omega$. 
Let $X$ represent a solution to the SPDE \eqref{SPDE0}, 
taking values in $\DD([0,T]; E_1)$. 
Due to the definition of the It{\^o} integral, 
the solution process $X$ must be progressively 
measurable with respect to the induced filtration 
$(\SF^{Y}_t)_{t \in [0, T]}$. 
Furthermore, since the system is autonomous, meaning 
that the only external influences come from the processes 
$\bW$, $\eta$, and the initial data $U_0$, 
the information contained 
in $(\SF^{Y}_t)_{t \in [0, T]}$ is sufficient to 
determine the process $X$ almost surely 
with respect to $\PP$. In particular, for any $t \in [0, T]$ and 
any bounded Borel measurable function 
$h:\DD([0, T]; E_1)\to \RR$, 
the following identity 
holds:\footnote{For two $\sigma$-algebras 
$\mathscr{G}_1$ and $\mathscr{G}_2$, the 
$\sigma$-algebra $\mathscr{G}_1 \lor \mathscr{G}_2$ 
is defined as $\sigma(\mathscr{G}_1
\cup \mathscr{G}_2)$. 
We have $\SF^{(X,Y)}_t = \SF^{X}_t \vee \SF^{Y}_t$, where 
$F^{Z}_t$, with $Z = (X,Y)$, $Z = X$, or $Z = Y$, are the 
induced filtrations (see Definition \ref{def38}).}
$$
\EE\left[ h(Y)\mid \SF^{(X,Y)}_t\right]
=\EE \left[ h(Y)\mid \SF^{Y}_t\right],
$$
recalling that a strong solution $X$ can be represented 
as $F(Y)$ for some function $F$ (see Definition \ref{strongsol}). 
This motivates the definition of \textit{temporal 
compatibility}, as given in Definition 2.1 of \cite{Kurtz:2014aa}, 
which is presented here for generic processes and is not 
specific to the $X$ and $Y$ associated 
with the above SPDE.

\begin{definition}\label{jcomp}
Let $(E_1,\mathcal{E}_1)$ and $(E_2,\mathcal{E}_2)$ 
be two Polish spaces, and let $X$ and $Y$ be defined on 
a probability space $(\Omega,\CF,\PP)$, taking values in 
$B_1:=\DD([0,T];E_1)$ and $B_2:=\DD([0,T];E_2)$, respectively. 
We say that the process $X$ is temporally compatible 
with $Y$,  if for every bounded $h\in \mathscr{B}(B_2)$ 
and for all times $t \in [0,T]$, the following equality 
holds almost surely:
$$
\EE \left[ h(Y)\mid 
\SF_t^{(X,Y)}
\del{\vee\SF_t^{Y}}\right]
=\EE \left[ h(Y) \mid 
\SF_t^{Y} \right],
$$
where, as before, $\bigl(\SF_t^{(X,Y)}\bigr)_{t\in[0,T]}$ 
and $\bigl(\SF_t^{Y}\bigr)_{t\in[0,T]}$ denote the 
filtrations induced by $(X,Y)$ and $Y$, respectively 
(see also Definition \ref{def38}).

We say that a \textit{probability measure 
$\mu$ is temporally compatible} if, for any pair of 
random variables $X$ and $Y$ with joint law $\mu$, the 
variable $X$ is temporally compatible with $Y$.
\end{definition}

\begin{remark}
If $X$ is temporally compatible 
with $Y$, then all relevant 
information is provided by the $\sigma$-algebra 
$(\SF_t^{Y})_{t\in [0,T]}$. In other words, 
once at time $t\in[0,T]$ $\SF_t^{Y}$ is known, 
knowing $\SF_t^{X}$ does not provide any additional 
information for calculating the expectation of $h(Y)$. 
Alternativly, at every time $t\in[0,T]$, knowing the history 
of both $X$ and $Y$ together 
(via $\SF_t^{(X ,Y)}$) 
does not improve our ability to predict $h(Y)$ 
beyond just knowing the history of $Y$ 
alone (via $\SF_t^{Y}$). 
If $Y$ has independent increments, then 
$X$ is temporally compatible with $Y$ if 
$Y(t + \cdot)-Y(t)$ is independent of 
$\CF^{(X,Y)}_{t}$ for all $t \in [0,T]$, 
see \cite[Lemma 3.2]{Kurtz:2007aa}.
\end{remark}

\begin{remark}
If the Stratonovich integral is considered, the solution $U$ 
of an SPDE at time $t \in [0,T]$ inherently depends on the future 
behavior of the process $t \mapsto \mathbb{W}_t$. Consequently, 
the solution to an SPDE interpreted in the Stratonovich sense 
does not satisfy the compatibility condition. Only the It\^{o} 
integral satisfies the compatibility condition described above. 
However, in practical applications, it is often possible to convert 
a Stratonovich-driven SPDE into an Itô-driven SPDE, allowing 
the above concept to be applied.
\end{remark}

\begin{remark}\label{rem:SGRC}
Recall that $\mathcal{S}_{\Gamma,\rho}$, 
which is defined in \eqref{eq:S-Gamma-rho}, 
represents the set of all 
joint solution measures $\mu_{X,Y} 
\in \mathcal{P}(B_1 \times B_2)$ that satisfy 
the constraint \eqref{solma} and for 
which $\mu(B_1 \times \cdot) = \rho$ ($\rho$ is 
the distribution of the input data $Y$). 
Since  temporal  compatibility is 
an additional property of the solution that must 
be satisfied, we introduce the so-called \textit{Kurtz set} 
$\CS_{\Gamma,\rho,\CT}$:
\begin{equation}\label{eq:S-Gamma-rho-C}
	\begin{split}
		&\text{$\CS_{\Gamma,\rho,\CT}$ 
		is the set of joint solution measures 
		$\mu_{X,Y}\in \CS_{\Gamma,\rho}$}
		\\ & 
		\text{that are temporally compatible 
		(see Definition \ref{jcomp}).}
	\end{split}
\end{equation}
\end{remark}

We have defined the set of solutions in \eqref{gammaus} 
using a countable dense set $\mathbb{Q}$ of 
times $t\in[0,T]$ and a countable dense set of test 
functions $\varphi$. Similarly, the concept of 
compatibility can be extended to a countable 
set or a sequence of random variables, as detailed in the 
following lemma. The proof of this lemma is 
similar to \cite[Lemma 5.7]{Bouard:2019ab} 
and is therefore not provided here.

\begin{lemma}\label{anothercomp}
Consider a cylindrical Wiener process $\bW$
and a Poisson random measure $\eta$
with intensity measure $\nu$, both defined 
over a probability space $(\Omega,\CF,\PP)$ as before.
Suppose $\bW$ and $\eta$ are adapted to a 
filtration $(\CF_t)_{t\in\INT}$ and let 
$U_0$ be an $\CF_0$-measurable 
$E_2$-valued random variable.  Then a sequence of 
$\DD([0,T];E_1)$-valued random variables 
$\left\{X_j: j\in\NN\right\}$ is 
temporally compatible with $Y=(\bW,\eta,U_0)$ 
if and only if
$$
\SF^{X_{I_1}}_t\lor\dots\lor
\SF^{X_{I_{{l}}}}_t\lor
\SF^{Y}_t
\,\,
\text{is $\PP$-independent 
of $\sigma(\mathscr{W}^t)$ 
and $\sigma(\eta^t)$},
$$
where $\mathscr{W}^t$ and $\eta^t$ 
are defined in \eqref{eq:wt} 
and \eqref{eq:etat}, respectively. 
Here, the notation $I_1, I_2, \dots, I_l$ refers to 
elements of an arbitrary index set $I$ from 
the power set of $\NN$. 
\end{lemma}


\section{Yamada-Watanabe uniqueness theory}
\label{sec:uniqueness}

In Section \ref{kframework}, we translated the 
abstract framework of Kurtz \cite{Kurtz:2014aa,Kurtz:2007aa} 
into the setting of SPDEs in variational form.
Now, we turn our attention to  
uniqueness of solutions. For SPDEs, one 
distinguishes between three types of uniqueness: 
\textit{(i) Uniqueness in law} 
means that if two solutions start from initial 
conditions with the same distribution, then the resulting 
solution processes will have the same distribution. 
\textit{(ii) Pathwise uniqueness} asserts that if 
two solutions are given on the same (but arbitrary) 
filtered probability space and start from 
the same initial condition, then
the two solutions are indistinguishable.
\textit{(iii) Strong uniqueness} states that, given a 
filtered probability space \eqref{probabilitysp} where 
the Wiener process, the Poisson random measure, 
and the initial condition are all defined, any 
solution is almost surely unique.

We will establish a key technical result 
(Lemma \ref{lem:transequal}), which asserts that if 
a process shares the same law as a variational 
solution---whether strong or weak, as per 
Definitions \ref{strongsol} and \ref{weaksol}---then 
this process also qualifies as a variational solution.
With the help of Lemma \ref{lem:transequal} 
and Kurtz's generalization of the 
Yamada-Watanabe theorem (see the upcoming 
Theorem \ref{theorem_abs}), we will obtain 
our primary uniqueness result 
(Theorem \ref{theoremuniquness}).

In the following two definitions, 
pathwise (pointwise) uniqueness and 
uniqueness in law (distribution) 
are formulated in the abstract setting 
(cf.~\cite[Definition 1.4]{Kurtz:2014aa}.

\begin{definition}\label{pointu}
Pathwise uniqueness is said to hold for 
the abstract equation \eqref{sol01} if, for any 
processes $X_1$, $X_2$, and $Y$ defined on the 
same probability space $(\Omega, \CF, \PP)$ and 
associated with the joint measures $\mu_{X_1,Y}$ 
and $\mu_{X_2,Y} \in \CS_{\Gamma, \rho, \CT}$ 
(cf.~Remark \ref{rem:SGRC}), 
respectively, it holds that
$$
\PP\left(\left\{X_1=X_2\right\}\right)=1.
$$
\end{definition}

\begin{definition}\label{pointu02}
Joint uniqueness in law (or weak joint uniqueness) 
is said to hold for the abstract 
equation \eqref{sol01} if $\CS_{\Gamma,\rho,\CT}$ 
contains at most one measure. Uniqueness in law (or weak 
uniqueness) holds if all solution measures 
$\mu\in \CS_{\Gamma,\rho,\CT}$ have the same marginal 
distribution on $B_1$ (cf.~Remark \ref{rem:SGRC}).
\end{definition}

In our context, uniqueness as described 
in Definition \ref{pointu} aligns with 
the standard notion of pathwise 
uniqueness defined below.

\begin{definition}\label{pathwiseu}
Whenever $\bigl(\MA,U^{(i)},\bW,\eta,U_0^{(i)}\bigr)$,
$i=1,2$, are two solutions to the SPDE \eqref{SPDE0} that
adheres to Definition \ref{def_solution}
and Assumption \ref{hyp_solution},
such that $\bW$ is a cylindrical Wiener process
evolving over $\WienH$, $\nu$ is the 
intensity measure of $\eta$,
$\rho_0$ is the law of $U_0^{(i)}$, $i=1,2$, and
$$
\PP\left(\left\{U^{(1)}(0)=U^{(2)}(0)\right\}\right)=1,
$$
then it holds that
$$
\PP\left(\left\{U^{(1)}(t)
=U^{(2)}(t)\right\}\right)=1,
\quad \forall t\in[0,T].
$$
\end{definition}

Often, one is only able in a first step to prove the 
existence of a probabilistic weak solution in the sense 
of Definition \ref{def_solution}, and, in a second  step 
is to verify pathwise uniqueness.  
However, pathwise uniqueness 
is often only achievable under additional 
regularity conditions on the solution process. 
These additional regularity properties are not 
inherent to the definition of a solution but 
are needed for proving uniqueness. 
They are introduced through two abstract mappings, 
$\theta^{\alpha_0}_0$ and $\theta^{\alpha_1}_1$, 
where $\alpha_0$ and $\alpha_1$ belong to some index sets. 
See the discussion leading to \eqref{calX}.

\begin{hypo}\label{hyp_solution}
Let $\bigl\{\theta^{\alpha_i}_i:\DD([0,T],E_1) 
\to [0,\infty] \mid \alpha_i \in A_i\bigr\}$, $i=0,1$, 
be two families of mappings, where $A_0$ and $A_1$ 
are index sets. The solution $U$ satisfies the 
condition $U \in \mathcal{X}$, 
where the set $\mathcal{X}$, defined 
in \eqref{calX}, depends on these 
families of mappings. 
\end{hypo}

Given a filtration $\BF=(\CF_t)_{t\in [0,T]}$ on 
a probability space $(\Omega,\CF, \PP)$, 
the augmented filtration $\BF^{\PP}
=(\CF_t^\mathbb{P})_{t\in [0,T]}$
is defined for each $t \ge 0$ as 
$\CF_t^\mathbb{P}=\sigma(\mathcal{F}_t 
\cup \mathcal{N})$, where $\mathcal{N}$ 
is the collection of all $\mathbb{P}$-null sets 
in $\mathcal{F}$. The augmented filtration is complete 
(i.e., it contains all null sets 
and is right-continuous). 

The next lemma demonstrates that although the 
tuple $(U,\bW, \eta, U_0)$, defined on a given 
filtered probability space $\MA$ as in 
\eqref{probabilitysp}, is not initially assumed to 
solve the SPDE, the equality of the laws, together with 
the existence of a solution on a different 
filtered probability space $\bar{\MA}$, implies 
that $(U, \bW, \eta, U_0)$ must satisfy the 
SPDE. Recall that pathwise uniqueness ensures that two 
solutions with the same driving 
noise (Wiener process and Poisson random measure) and 
initial conditions are indistinguishable. 
Therefore, when pathwise uniqueness is known, even 
if the solutions are defined on different 
filtered probability spaces, they can be 
related through their distributions. 
This connection allows us to conclude the existence of 
a unique solution on the original space $\MA$.

\begin{lemma}\label{lem:transequal}
Consider a filtered probability space $\MA=(\Omega,\CF,\BF,\PP)$
with filtration $\BF=(\mathcal F_t)_{t\in \INT}$,
along with the following additional elements:
\begin{itemize}
	\item $U$, a $\Bbb D({[0,T]};E_1)$-valued random variable,
	
	\item $\bW$, a cylindrical Wiener process on
	$\WienH$, represented as \eqref{eq:wiener_represent},

	\item $\eta$, a random variable with values in
	$M_{{\Bbb N}}(\{S_n \times\INT\})$, and
	\item $U_0$, a $\CF_0$-measurable 
	$E_2$-valued random variable,
\end{itemize}
all defined on $\MA$. In addition,  suppose that 
for any $t\in [0,T]$, $U(t)$, 
$\mathscr{W}_t$, and $\eta_t$ are $\mathcal F_t$-measurable, 
and that $\mathscr{W}^t$ and  $\eta^t$ 
are independent of $\mathcal F_t$ (see 
Definition \ref{def:information} and Lemma \ref{adapttimhom}).  

Suppose there exists a solution
$$
(\bar{\MA},\bar U,\bar{\bW},\bar\eta,\bar U_0),
$$
where $\bar{\MA}:=(\bar{\Omega},\bar{\CF},
\bar{\BF},\bar{\PP})$, $\bar{\BF}=(\bar{\CF}_t)_{t\in [0,T]}$,
is a potentially different filtered probability space,
and this solution to the SPDE \eqref{SPDE0} adheres to
the conditions outlined in Definition \ref{def_solution}
and Assumption \ref{hyp_solution}, such that
the law of $\bigl(U,\bW,\eta,U_0\bigr)$ 
coincides with the law of 
$\bigl(\bar U,\bar{\bW},\bar\eta,\bar U_0\bigr)$
on $\DD({[0,T]};E_1)\times  C_b(0,T;\WienH)\times
M_{{\Bbb N}}(\{S_n \times\INT\}) \times E_2$.

\medskip

Then the tuple $(\MA^{\PP},U,\bW, \eta,U_0)$, 
where $\MA^{\PP}:=(\Omega,\CF,\BF^{\PP},\PP)$, 
is a solution to the SPDE \eqref{SPDE0}, 
according to Definition \ref{def_solution} 
and Assumption \ref{hyp_solution}.
\end{lemma}

\begin{proof}
The proof is similar to the proof of  
Lemma 4.4 in \cite{Bouard:2019ab} with the
modification of incorporating the Wiener process
and  the variational setting.
Given the equality of laws and the premise that
$(\bar{\MA},\bar U,\bar {\bW},\bar\eta,\bar U_0)$
constitutes a solution, the procedure of verifying that 
$(\MA^{\PP},U,\bW, \eta)$ meets the criteria 
(i) through (vii) outlined in Definition \ref{def_solution} 
and conforms to Assumption \ref{hyp_solution}
follows a routine argument. Consequently, 
our focus below is to verify condition (viii). This 
involves demonstrating that the triple 
$(U,\bW, \eta,U_0)$ satisfies the 
specified SPDE \eqref{SPDE0} 
in the sense of \eqref{SPDE}.

Consider a process $U$ that is progressively 
measurable over $\MA$ and c\`adl\`ag in the space $E_1$, meaning 
that for each time $t \geq 0$, $U(t)$ takes values in $E_1$, with 
well-defined left limits, and satisfies $\lim_{s \downarrow t} 
U(s) = U(t)$ in $E_1$ almost surely with respect to $\PP$.
Suppose this process satisfies the condition given 
in \eqref{finiteintegrals}. 
For every test function $\varphi \in V$ 
(see \eqref{eq:Gelfand}) and time $t\in \INT$, 
we introduce a nonlinear 
mapping $\CK^\varphi_{\MA}$ as follows:
\begin{align*}
	\CK^\varphi_{\MA}\bigl(U,\bW,\eta,U_0\bigr)(t)
	& := \la U_0,\varphi\ra
	+ \int_0^t \langle b(s,U(s)),\varphi \rangle
	\\ & \qquad
	+ \int_0^t \sum_{{k}=1}^\infty 
	\langle \sigma(s,U(s))[h_{k}],
	\varphi\rangle\,d{\beta^{k}}(s)
	\\ & \qquad
	+ \int_0^t\int_S\langle c(s,x,U(s)),
	\varphi\rangle \,\tilde\eta(dx,ds).
\end{align*}
Note that ${\CK^\varphi_{\MA}}(V,\bW,\eta,U_0)$ is contingent
upon $\bW$ through the Hilbert space $\WienH$ and
on $\eta$ via its compensator. Consequently, it also depends
on the probability measure $\PP$.
The objective is to demonstrate that if, for 
all $ \varphi \in V $ and $ t \in [0, T] $,
$$
\bar{\PP}\Bigl( \CK^\varphi_{\bar{\MA}}
\bigl(\bar{U},\bar{\bW},\bar{\eta},\bar{U}_0\bigr)(t)
-\bar{U}(t)=0\Bigr)=1,
$$
then it necessarily follows that, for all 
$\varphi \in V$ and $t \in [0, T]$,
\begin{equation}\label{eq:PK-eqn}
	\PP\Bigl( \CK^\varphi_{\MA}
	\bigl(U,\bW,\eta,U_0\bigr)(t)-U(t)=0\Bigr)=1.
\end{equation}

To establish this, we examine each 
component of the operators 
$\CK^\varphi_{\MA}\bigl(U,\bW,\eta,U_0\bigr)$ 
and $\CK^\varphi_{\bar{\MA}}(\bar{U},\bar{\bW},
\bar{\eta},\bar{U}_0)$, demonstrating their 
equivalence in distribution. This equivalence directly 
implies that \eqref{eq:PK-eqn} holds, 
or equivalently, the tuple $(\MA,U,\bW, \eta,U_0)$ serves 
as a solution to the SPDE \eqref{SPDE0}.

Initially, it is evident that the components of
${\CK^\varphi_{\MA}}(\bar{U},\bar{\bW},\bar{\eta},\bar{U}_0)$
adhere to \eqref{finiteintegrals}.
Given that the laws of $(U, \bW, \eta,U_0)$ and
$(\bar{U}, \bar{\bW}, \bar{\eta},\bar{U}_0)$ are equal, and
considering the functions $b$, $\sigma$, and $c$ 
are measurable, it
naturally follows that the components of 
$\CK^\varphi_{\MA}\bigl(U,\bW,\eta,U_0\bigr)$
equally fulfill \eqref{finiteintegrals}. 
Recall also that the initial data $U_0$, $\bar U_0$ 
are given and share the same law.

Define the following real-valued processes:
$$
\overline{\mathfrak{b}}_\varphi(s,\bar \omega)
:= \left\langle b(s,\bar U(s,\bar \omega)), \varphi \right\rangle,
\quad s\in [0,T],\,\, \bar \omega \in \bar \Omega,
$$
and
$$
\mathfrak{b}_\varphi(s,\omega) 
:=\left \langle b(s, U(s,\omega)),
\varphi \right\rangle, \quad s\in [0,T],
\,\, \omega \in \Omega.
$$
Given the underlying assumptions, we have 
$\Law(U) = \Law(\bar U)$ in
the intersection space $\mathbb{D}([0,T];E)
\cap L^2(0,T;V)$. Consequently, 
$\mathfrak{b}_\varphi$ and 
$\overline{\mathfrak{b}}_\varphi$
share identical laws on $\mathbb{D}([0,T];\RR)$ 
(for every $\varphi\in V$). Leveraging 
\cite[Theorem 8.3]{Ondrejat:2004aa}, 
it follows that for any 
time $t\in[0,T]$, the laws
$$
\Law\left(\bar U,
\int_0^t\overline{\mathfrak{b}}_\varphi(s)\,ds\right)
\quad\text{and}\quad
\Law\left(U,
\int_0^t \mathfrak{b}_\varphi(s)\,ds\right)
$$
are identical on the space 
$\DD([0,T];E)\times \RR$ 
(across all $\varphi\in V$).

Define the following real-valued processes:
$$
\overline{\mathfrak{s}}_\varphi^{k}(s,\bar \omega)
:=\left \langle \sigma(s,\bar U(s,\bar \omega))
[h_{k}], \varphi \right\rangle,
\quad s \in [0, T], \,\,  
\bar \omega \in \bar \Omega
$$
and
$$
\mathfrak{s}_\varphi^{k}(s, \omega) :=
\left\langle \sigma(s, U(s, \omega))[h_{k}], 
\varphi \right\rangle,
\quad s \in [0, T], 
\, \, \omega \in \Omega.
$$
The two processes 
$\left\{{\overline{\mathfrak{s}}_\varphi^{k}}(s) \mid s \in [0,T]\right\}$
and $\left\{{\overline{\mathfrak{s}}_\varphi^{k}}(s) \mid s\in [0,T]\right\}$
are adapted to the filtrations $(\bar{\CF}_{s})_{s \in [0,T]}$
and $(\CF_s)_{s \in [0,T]}$, respectively.
Besides, the laws of $(U, \bW)$ and
$(\bar{U}, \bar{\bW})$ are identical.
Drawing upon \cite[Theorem 8.6]{Ondrejat:2004aa}, we
infer that for any $t \in [0, T]$, the laws
$$
\Law\left(\bar U, \int_0^t \sum_{{k}=1}^\infty
\overline{\mathfrak{s}}^{k}_\varphi(s) 
\,d{\beta^{k}}(s)\right)
\quad\text{and}\quad
\Law\left(U, \int_0^t \sum_{{k}=1}^\infty
\mathfrak{s}^{k}_\varphi(s) 
\,d{\beta^{k}}(s)\right)
$$
are equal on the space 
$\bigl(\DD([0,T];E_1) \cap 
L^2(0,T;V)\bigr) \times \RR$ 
(for all $\varphi \in V$).

In the final step, define the 
processes through
$$
\overline{\mathfrak{c}}_\varphi(s,\bar \omega)
:= \left\langle c(s,\bar U(s,\bar \omega)),
\varphi \right\rangle, \quad s\in [0,T], \,\,
\bar \omega \in \bar\Omega
$$
and
$$
\mathfrak{c}_\varphi(s,\omega)
:=\left\langle c(s, U(s,\omega)),
\varphi\right\rangle,
\quad s\in [0,T], \,\, \omega \in \Omega,
$$
which are adapted to the filtrations 
$(\bar \CF_{s})_{s\in[0,T]}$
and $(\CF_{s})_{s\in[0,T]}$, respectively. 
By the assumption \eqref{finiteintegrals} 
(and the equality of laws), 
$\bar{\mathfrak{c}}_\varphi$ and 
$\mathfrak{c}_\varphi$ belong a.s.~to $L^p(\INT;\RR)$.
Proposition B.1 in \cite{Bouard:2019ab} 
guarantees the progressive measurability of 
$\bar{\mathfrak{c}}_\varphi$ and $\mathfrak{c}_\varphi$. 
This is explained by employing a sequence of 
shifted Haar projections, which 
approximates (in $L^p$) the original processes 
with simpler processes that 
possess the necessary measurability. 
Define the stochastic integrals
$$
\bar \CI_\varphi(t)
=\int_{0}^t \int_S \left\langle c(s,\bar U(s),z),
\varphi \right\rangle 
\, \tilde{\bar \eta}(dz,ds), 
\quad t\in[0,T],
$$
and
$$
\CI_\varphi(t) 
= \int_{0}^t \int_S
\left\langle {c(s,U(s),z)},\varphi \right\rangle
\, \tilde \eta(dz,ds), \quad t\in[0,T].
$$
Applying \cite[Theorem A.1]{Bouard:2019ab}, we conclude
that $(\bar \CI_\varphi(t),\bar U,\bar \eta)$ and
$(\CI_\varphi(t),U,\eta)$ possess the same law on
$\RR\times \DD([0,T];E)\times
\CMM_{\NN}(\{S_n \times\INT\})$.

Hence, the lemma is proved.
\end{proof}

In the abstract setting, the following 
generalisation of the Yamada-Watanabe 
theorem holds \cite[Theorem 1.5]{Kurtz:2014aa}.

\begin{theorem}\label{theorem_abs}
The following are equivalent:
\begin{itemize}
	\item $\CS_{\Gamma,\rho,\mathcal{T}}\not = \emptyset$, see 
	\eqref{eq:S-Gamma-rho-C}, 
	and pathwise uniqueness holds.
	
	\item There exists a strong solution 
	(Definition \ref{strongsol}) and joint uniqueness 
	in law (Definition \ref{pointu02}) holds.
\end{itemize}
\end{theorem}

The abstract Theorem \ref{theorem_abs} 
can be tailored to our specific SPDE context. 
We now outline the detailed setup 
and reformulate Theorem \ref{theorem_abs} 
within this framework. Consider 
a filtered probability space $\MA$ 
(see \eqref{probabilitysp}) with a cylindrical 
Wiener process $\bW$ on $\CH$ (see Section \ref{defcyw}), 
a Poisson random measure $\eta$ on $S$ 
(see Section \ref{deflev}), and an initial condition 
$U_0 \in E_2$, where $U_0$ is $\CF_0$-measurable. Here, 
$\eta$ and $\bW$ are independent and adapted 
to the filtration $(\CF_t)_{t \in [0, T]}$. 
Define $Y = (\bW, \eta, U_0)$ on $B_2$ 
(see \eqref{defB2}), and let 
$X = U$ denote a solution to the 
SPDE \eqref{SPDE0} within $B_1$ (see \eqref{eq:B1-def}).
The mappings $\Omega \ni \omega \mapsto X(\omega)$ 
and $\Omega \ni \omega \mapsto Y(\omega)$ 
induce two filtrations  $(\SF^X_t)_{t \in [0,T]}$ 
and $(\SF^Y_t)_{t \in [0,T]}$ on 
$\MA$ (see Definition \ref{jcomp}). 
Here, the notion of temporal compatibility 
becomes relevant.

To clarify further, let us specify $\mathcal{B}^{B_i}_t$ for $i =1,2$, 
where $i = 1$ corresponds to $X$, and $i = 2$ corresponds to $Y$. 
For the solution process $X$, as defined in Definition \ref{def38}, 
the $\sigma$-algebra $\CB_t^{B_1}$ is generated by the 
coordinate map $\pi_s: z \in B_1 \mapsto z(s) \in 
E_1$ for $s \leq t$, $t \in [0, T]$. 
The input data $Y$ takes values in the path space 
$B_2 = C([0,T];\CH) \times M_\NN(\{S_n \times [0,T]\}) \times E_2$. 
On this path space $B_2$, we define the $\sigma$-algebra 
generated by the input data $Y$ as 
$\CB_t^{B_2} = \sigma(\bW_t) \otimes \sigma(R_t) \otimes \sigma_B(U_0)$, 
where $\bW_t$ and $R_t$, for $t \in [0, T]$, 
denote the canonical restriction 
mappings defined as follows:
\begin{itemize}
	\item $\bW_t$ restricts a Wiener process $\bW$ 
	to the interval $[0, t]$, i.e., $\bW_t$, mapping 
	from 
	$C([0, T]; \WienH)$ to itself, is defined by 
	$\bW \mapsto \bW \mathds{1}_{[0, t]}$.

	\item $R_t$ restricts the Poisson random measure $\eta$ to 
	the set $\Sn \times (0, t]$, i.e., $R_t$, which maps 
	from $M_{\NN}\left(\{S_n \times \INT\}\right)$ to 
	itself, is defined as
	\begin{align*}
		&R_t:\eta\left( A \times I \right) 
		\mapsto \eta\left(A \times 
		\left( I \cap (0, t] \right)\right), 
		\\ & 
		\text{where $A \in \mathscr{B}(S_n)$ for 
		some $n \in \NN$ and $I \in \mathscr{B}([0, T])$.}
	\end{align*}
\end{itemize}
Since the initial condition $U_0$ is known at 
the outset and remains unchanged thereafter, we 
have $\mathcal{B}_0^{B_2} = \sigma(U_0)$. 
Moreover, as $U_0$ is $\CF_0$-measurable by definition,
we set $\CF_0 := \sigma(\{U_0^{-1}(B) : B \in \mathscr{B}(E_2)\})$, 
ensuring that $U_0$ is measurable with respect to 
both $\CF_0$ and $\mathscr{B}(E_2)$.

Let $\rho_\bW$ denote the law of the 
cylindrical Wiener process $\bW$ over 
$C([0,T];\WienH)$. Similarly, let $\rho_\nu$ 
represent the law of the Poisson random measure 
$\eta$ equipped with the L{\'e}vy measure 
$\nu$ on $M_{\NN}(\{S_n \times\INT\})$. 
Furthermore, let $\rho_0$ symbolize the 
law of the $\CF_0$-measurable random variable 
$U_0$ over $E_2$. Given the independence 
of $\bW$, $\eta$, $U_0$, we proceed to define 
their joint law $\rho$ as the product measure
\begin{equation}\label{eq:joint-law}
	\rho=\rho_\bW \times \rho_\nu \times \rho_0.
\end{equation}

Let $V_d=\{\varphi _k:k\in\NN\}$ be a dense 
countable subset of $V$ (see \eqref{eq:Gelfand}), 
and recall the set $\Gamma$ of mappings 
given by \ref{gammaus} and \eqref{abstracteqnew}. 
The SPDE \eqref{SPDE0} (via \eqref{SPDE}) defines the mappings in
$\Gamma$ in \eqref{gammaus}, \eqref{abstracteqnew}. 
Specifically, given the 
filtered probability space $\MA$ described in 
\eqref{probabilitysp}, we define $\Gamma$ by 
associating a real number 
$\Gamma_{\varphi,t}(U,(\bW,\eta,U_0))$ with an 
arbitrary process $U$ and any triplet $(\bW,\eta,U_0)$ 
on $\MA$, in accordance with \eqref{abstracteqnew}.

Using the notation introduced above, we can now 
reformulate the SPDE \eqref{SPDE0} as a property of a 
solution measure on the path spaces $B_1$ and $B_2$, 
expressed in the format \eqref{sol01}. To be more precise,  let 
$X$ represent the solution $U$, and $Y$ 
denote the input data $(\bW, \eta, U_0)$,  then $X$ is 
a solution if   $\PP\lk( \Gamma(X,Y) = 0\rk)=1$.
Our objective is to identify a solution measure
$\mu\in\mathcal{S}_{\Gamma,\rho,\mathcal{T}}$ 
on $B_1 \times B_2$ (see \eqref{eq:S-Gamma-rho-C}) 
that adheres to the condition
\begin{equation}\label{eq:solm-new}
	\int_{B_1\times B_2} 
	\left|\Gamma_\varphi(x,y)(t)\right|
	\, \mu(dx,dy) = 0,
\end{equation}
for all $\varphi\in V_d$ and 
$t\in \QQ_T$ (compare with \eqref{solma}). 
Pathwise uniqueness, however, necessitates further
constraints on the solution beyond those 
defined in \eqref{eq:solm-new}.
The additional properties and regularity specified in 
Assumption \ref{hyp_solution} must be conveyed as 
a condition on the solution measure. Specifically, the 
solution measure $\mu$ is required to satisfy:
\begin{align}
	& \mu\left(\bigl\{ (x,y)\in B_1\times B_2:
	\theta_0^{\alpha_0}(x)<\infty\bigr\}\right) = 1,
	\quad
	\forall \alpha_0 \in A_0,
	\label{cons1}
	\\ &
	\int_{B_1\times B_2} \theta_1^{\alpha_1}(x)
	\, \mu(dx\times dy) < \infty,
	\quad \forall \alpha_1\in A_1.
	\label{cons2}
\end{align}

\begin{remark}

Although we will not make use of this, 
the set of solution measures $\mu$  
on $B_1\times B_2$ satisfying \eqref{cons1} and 
\eqref{cons2} is convex 
\cite[Lemma 3.8.]{Kurtz:2007aa}.
\end{remark}

We use the notation $\Gamma^{\theta}$ 
to represent the combination of constraints 
specified in \eqref{eq:solm-new}, \eqref{cons1}, 
and \eqref{cons2}.  The superscript $\theta$ 
indicates the additional constraints \eqref{cons1} 
and \eqref{cons2}. When $\mu$ denotes a Borel 
probability measure that is the joint law of the 
random vector $\bigl(U,(\bW,\eta,U_0)\bigr)$, we 
say that $\mu$ adheres to the Kurtz convexity 
constraint $\Gamma^{\theta}$.

Following \cite[p.~958]{Kurtz:2007aa}, we now 
explicitly define the Kurtz set 
$\CS_{\Gamma^\theta,\rho,\CT}$ 
(see \eqref{eq:S-Gamma-rho-C}) within 
the present SPDE context. 
This set consists of all solution measures 
that meet the specified criteria. 
Indeed, a probability measure 
$\mu \in \mathcal{P}(B_1 \times B_2)$ belongs to 
$\CS_{\Gamma^\theta,\rho,\CT}$ if and only if it 
satisfies the following conditions:
\begin{itemize}
	\item $\mu$ is in compliance with the convexity
	constraint $\Gamma^{\theta}$, 
	see \eqref{eq:solm-new}, \eqref{cons1}, and \eqref{cons2}, 
	and  $\mu$ is temporally compatible (see 
	Definition \ref{jcomp});
	
	\item For every $A$ in the Borel $\sigma$-algebra
	$\mathscr{B}(B_2)$, $\mu(B_1 \times A)=\rho(A)$,
	where $\rho$ is the joint law of $(\bW,\eta,U_0)$,
	see \eqref{eq:joint-law}.
\end{itemize}
An alternative description of the Kurtz set
$\CS_{\Gamma^{\theta},\rho,\CT}$ is
as follows: $\mu \in \mathcal{P}(B_1 \times B_2)$
belongs to $\CS_{\Gamma^\theta,\rho,\CT}$ if 
and only if the following criteria are met:
\begin{itemize}
	\item there exists a solution 
	tuple $\bigl(\bar{\MA}, \bar{U}, \bar{\bW},
	\bar{\eta},\bar{U}_0\bigr)$ 
	to the SPDE \eqref{SPDE0}, where $\bar{\MA}$ is 
	a filtered probability space, satisfying 
	Definition \ref{def_solution} and 
	Assumption \ref{hyp_solution}, and 
	\begin{itemize}	
		\item $\bar \bW$ is a cylindrical Wiener 
		process evolving over the Hilbert space 
		$\WienH$ (see Section \ref{defcyw});

		\item $\bar\eta$ is a Poisson random measure
		with the intensity measure $\nu$ 
		(see \eqref{eq:joint-law} 
		and Section \ref{def-PRM});
	
		\item $\bar U_0$ has distribution 
		$\rho_0$ (see \eqref{eq:joint-law}); 
	\end{itemize}
	
	\item $\mu$ is the joint law of
	$\bigl(\bar U,(\bar {\bW},{\bar\eta},\bar U_0)\bigr)$
	on $B_1 \times B_2$ (see \eqref{eq:B1-def} 
	and \eqref{defB2}).
\end{itemize}

Given the new constraints, the definitions of temporal 
compatibility and pointwise uniqueness must be adjusted 
accordingly.  In particular, the temporal compatibility condition 
requires that, for $t \in [0, T]$, the solution $\bar{U}(t)$ 
is independent of $\bar{\mathscr{W}}^t$ and $\bar{\eta}^t$ 
(see Definition \ref{def:information}) and is measurable 
with respect to $\bar{\mathscr{W}}_t$, $\bar{\eta}_t$, 
and $\sigma(\bar{U}_0)$.

\medskip

Our main result is presented in the next theorem, which 
establishes a connection between the existence 
of a weak solution (Definition \ref{weaksol}) with pathwise 
uniqueness (Definition \ref{pointu}) and the 
existence of a unique strong solution (see Definition \ref{strongsol}). 
Using the results developed earlier in this section, the 
proof follows an adaptation of \cite{Bouard:2019ab}, with 
the only modification being the inclusion of the Wiener process. 
Due to the similarity in reasoning, a detailed proof is omitted.

\begin{theorem}\label{theoremuniquness}
Consider a Gelfand triple $(V,H,V')$, see \eqref{eq:Gelfand}, 
along with Banach spaces $E_1$ and $E_2$. Here, $E_2$ is 
continuously embedded into $E_1$, and $V$ is 
continuously embedded into $E_1$.  
Additionally, $E_1$ is a UMD space of type 2. 
Let $\rho_0$ be a Borel probability measure on 
$E_2$, see \eqref{eq:joint-law}. 
Suppose $\WienH$ is a Hilbert space, $\nu$ is an
intensity measure over a Polish space $(S,\CS)$, 
see \eqref{eq:intensity0} and \eqref{eq:intensity-measure}, 
and assume that
\begin{itemize}
	\item there exists a solution
	$\bigl(\bar{\MA},\bar U,
	\bar{\bW},\bar\eta,\bar U_0\bigr)$
	to the SPDE \eqref{SPDE0} that
	adheres to Definition \ref{def_solution}
	and Assumption \ref{hyp_solution},
	such that $\bar{\bW}$ is a cylindrical Wiener process
	evolving over $\WienH$, see \eqref{eq:wiener_represent},
	$\nu$ is the intensity measure of $\bar\eta$, 
	and $\rho_0$ is the probability law of $\bar U_0$;

	\item pathwise uniqueness, as defined 
	in Definition \ref{pathwiseu}, is satisfied.
\end{itemize}

Then there exists a Borel measurable mapping
$$
F:C([0,T];\WienH)\times M_{\NN}(\{S_n \times\INT\})
\times E_2 \to\Bbb D([0,T];E_1),
$$
depending on $\WienH$, {$\nu$} and $\rho_0$, such that
\begin{itemize}
	\item if $(\MA,U,\bW, \eta,U_0)$ is a solution
	to the SPDE \eqref{SPDE0}, in the sense of
	Definition \ref{def_solution} 
	and Assumption \ref{hyp_solution},
	such that $\bW$ is a cylindrical 
	Wiener process on $\WienH$ with
	the representation \eqref{eq:wiener_represent}, 
	$\nu$ is the intensity measure of $\eta$, 
	and $\rho_0$ is the law of $U(0)$, then
	$$
	U=F(\bW,\eta,U_0) \quad 
	\text{$\PP$-almost surely},
	$$
	and $U$ is progressively measurable 
	with respect to the $\PP$-augmentation of the 
	filtration (cf.~Definition \ref{def:information})
	$$
	\qquad\qquad 
	\Bigl(\sigma\bigl(\left\{\mathscr{W}_t(h)
	:h\in\mathcal{H}\right\}\bigr),
	\sigma\bigl(\left\{\eta_t(V):
	V\in\CS\otimes\mathscr{B}([0,T])\right\}\bigr),
	\sigma(U_0)\Bigr)_{t\in\INT};
	$$
	
	\item if $\MA:=(\Omega,\CF,(\CF_t)_{t\in [0,T]},\PP)$ 
	is a filtered probability space, 
	$U_0$ is an $E_2$-valued $\CF_0$-measurable 
	random variable with law $\rho_0$, $\bW$ is a 
	cylindrical Wiener process on $\WienH$ with the
	representation \eqref{eq:wiener_represent}, and  $\eta$ is
	a time-homogeneous Poisson random measure 
	with intensity $\nu$, then
	$$
	U=F(\bW, \eta,U_0)
	$$
	is adapted to the augmented filtration
	$(\CF^{\PP}_t)_{t\in\INT}$, 
	$U(0)=U_0$, $\PP$-a.s., and
	$$
	\bigl(\MA^{\PP},U,\bW, \eta,U_0\bigr), 
	\quad \text{where $\MA^{\PP}
	:=\bigl(\Omega,\CF,(\CF^{\PP}_t)_{t\in\INT},\PP\bigr)$},
	$$
	constitutes a solution to the SPDE \eqref{SPDE0}
	in the sense of Definition \ref{def_solution},
	fulfilling Assumption \ref{hyp_solution}.
\end{itemize}
\end{theorem}

We briefly outline the proof of Theorem \ref{theoremuniquness}. 
From the assumptions of Theorem \ref{theoremuniquness}, it 
follows that the Kurtz set $\CS_{\Gamma^\theta,\rho,\CT}$ is non-empty. 
Furthermore, pathwise uniqueness is assumed. Applying 
Theorem \ref{theorem_abs}, we conclude the 
existence of a strong solution and joint uniqueness in law. 
In particular, there exists a Borel measurable function 
$$
F:C_b([0,T]; \WienH) \times M_{\NN}(\{S_n \times \INT\}) 
\times E_2 \to \DD([0,T]; E_1),
$$
such that, for a given filtered probability space 
$\MA:=(\Omega, \CF, \mathbb{F}, \PP)$ with 
$\mathbb{F} = (\CF_t)_{t \in [0,T]}$, where
$U_0$ is an $E_2$-valued, $\CF_0$-measurable 
random variable with law $\rho_0$, 
$\bW$ is a cylindrical Wiener process on $\WienH$ 
with the representation \eqref{eq:wiener_represent}, and 
$\eta$ is a time-homogeneous Poisson 
random measure with intensity $\nu$, 
the solution $U$ is given $\PP$-almost surely 
by $U = F(\bW, \eta, U_0)$. Moreover, due to pathwise 
uniqueness and Lemma \ref{lem:transequal}, if the triplet 
$(\bW, \eta, U_0)$ is defined on any 
other filtered probability space satisfying 
the constraints in Lemma \ref{lem:transequal}, then the solution 
$U$ (given by $F(\bW, \eta, U_0)$) on the specified 
probability space is unique. Hence, we conclude that 
there exists a unique strong solution.

\begin{remark}
A consequence of 
Theorem \ref{theoremuniquness} is
that if $\bigl(\MA^{(i)},U^{(i)},\bW^{(i)},
\eta^{(i)},U_0^{(i)}\bigr)$,
where $\MA^{(i)}=\bigl(\Omega^{(i)},\CF^{(i)},
\BF^{(i)},\PP^{(i)}\bigr)$, $\BF^{(i)}=(\CF^{(i)}_t)_{t\in\INT}$, 
$i=1,2$, are two solutions to the SPDE \eqref{SPDE0} 
in the sense of Definition \ref{def_solution} 
and Assumption \ref{hyp_solution}, such that $\bW^{(i)}$ 
are cylindrical Wiener processes evolving over 
$\WienH$, $\nu$ is the intensity measure 
of $\eta^{(i)}$, and $\rho_0$ is the 
law of $U_0^{(i)}$, $i=1,2$, then 
we have uniqueness in law:
$$
\Law\left(U^{(1)},\bW^{(1)},\eta^{(1)}\right)
\quad \text{coincides with} \quad
\Law\left(U^{(2)},\bW^{(2)},\eta^{(2)}\right).
$$
\end{remark}

\appendix

\section{The Skorohod space}

For an introduction to the Skorokhod space, 
we refer the reader to 
\cite{Billingsley:1999aa,Ethier:1986aa,Jacod:2003aa}. 
In this section, we recall only a few definitions  
that are essential for our work.
Let $(Y, |\cdot|_Y)$ be a separable Banach space. 
The space $\DD(0,T; Y)$ denotes the set 
of all right-continuous functions $x: [0,T] \to Y$ 
with left-hand limits. Let $\Lambda$ denote the 
class of all strictly increasing continuous functions 
$\lambda:[0,T] \to [0,T]$ such that 
$\lambda(0) = 0$ and $\lambda(T) = T$. 
Clearly, any $\lambda \in \Lambda$ is a 
homeomorphism of $[0,T]$ onto itself. 
We now define the Skorohod topolgy. 
First, let
$$
\norm{\lambda}_{\log}
:=\sup_{\substack{t,s \in [0,T] \\ t \neq s}} 
\left| \log \left( \frac{\lambda(t) 
- \lambda(s)}{t-s} \right) \right|
\sim 
\operatorname*{ess\,sup}_{t \in [0,T]}
\bigl| \log\left( \lambda'(t) \right) \bigr|.
$$
Introducing the function class
$$
\Lambda_{\log} := \left\{ \lambda \in \Lambda:
\norm{\lambda}_{\log} < \infty \right\},
$$
the 
 metric $d_0$ between 
$x$ and $y$ in $\DD(0,T;Y)$ 
is given by
$$
d_0(x, y) := \inf_{\lambda \in \Lambda_{\log}}
\left\{ \norm{\lambda}_{\log}
\wedge \sup_{t \in [0,T]} 
\abs{x(t) - y(\lambda(t))}_Y \right\}.
$$
The metric $d_0$ generates the Skorohod topology 
on $\DD([0,T];Y)$. Moreover, the space $\DD([0,T];Y)$, 
equipped with the metric $d_0$, is 
a complete and separable metric space.

The metric $d_0$ is particularly useful for comparing 
discontinuous functions because it enables small 
adjustments in the time axis to better align 
their discontinuities. This flexibility ensures 
that even if two functions exhibit similar behavior 
but with slight variations in the timing of their jumps, 
the $d_0$ metric accurately reflects their closeness, 
while the $L^\infty$ metric does not.


\bibliographystyle{abbrv}

\end{document}